\documentclass[11pt]{article}
\usepackage{amsfonts,amssymb,amsmath}
\usepackage{dsfont}
\usepackage{indentfirst}
\usepackage{amsmath,amssymb,latexsym,mathrsfs,amsthm}
\usepackage[usenames,dvipsnames,svgnames,table]{xcolor}
\usepackage{indentfirst}
\usepackage{amsmath}
\usepackage{amssymb}
\usepackage{hyperref}
\usepackage{enumitem}

\theoremstyle{definition} 

\newtheorem{Def}{Definition}[section]

\theoremstyle{definition}
\newtheorem{Exe}[Def]{Example}
\newtheorem{Obs}[Def]{Remark}

\theoremstyle{plain}
\newtheorem{Pro}[Def]{Proposition}
\newtheorem{Cor}[Def]{Corollary}
\newtheorem{Teo}[Def]{Theorem}
\newtheorem{Lem}[Def]{Lemma}

\newtheorem{Not}[Def]{Notation}

\newcommand{\<}{\langle}
\renewcommand{\>}{\rangle}

\newcommand{\T}{\mathbb{T}}
\newcommand{\R}{\mathbb{R}}
\newcommand{\M}{\mathscr{M}}
\newcommand{\N}{\mathbb{N}}
\newcommand{\Q}{\mathbb{Q}}

\newcommand{\C}{\mathbb{C}}

\newcommand{\D}{\mathscr{D}}
\newcommand{\E}{\mathcal{E}}

\newcommand{\Z}{\mathbb{Z}}

\newcommand{\ds}{\displaystyle}
\newcommand\nnfootnote[1]{%
  \begin{NoHyper}
  \renewcommand\thefootnote{}\footnote{#1}%
  \addtocounter{footnote}{-1}%
  \end{NoHyper}
}
\numberwithin{equation}{section}

\usepackage{color}

\definecolor{Red}{rgb}{1.00, 0.00, 0.00}
\definecolor{DarkGreen}{rgb}{0.00, 1.00, 0.00}
\definecolor{Blue}{rgb}{0.00, 0.00, 1.00}
\definecolor{Cyan}{rgb}{0.00, 1.00, 1.00}
\definecolor{Magenta}{rgb}{1.00, 0.00, 1.00}
\definecolor{DeepSkyBlue}{rgb}{0.00, 0.75, 1.00}
\definecolor{DarkGreen}{rgb}{0.00, 0.39, 0.00}
\definecolor{SpringGreen}{rgb}{0.00, 1.00, 0.50}
\definecolor{DarkOrange}{rgb}{1.00, 0.55, 0.00}
\definecolor{OrangeRed}{rgb}{1.00, 0.27, 0.00}
\definecolor{Black}{rgb}{0.00, 0.00, 0.00}
\definecolor{dark-magenta}{rgb}{.5,0,.5}
\definecolor{myblack}{rgb}{0,0,0}
\definecolor{DeepPink}{rgb}{1.00, 0.08, 0.57}
\definecolor{DarkViolet}{rgb}{0.58, 0.00, 0.82}
\definecolor{SaddleBrown}{rgb}{0.54, 0.27, 0.07}
\definecolor{darkgray}{gray}{0.5}
\definecolor{lightgray}{gray}{0.75}

\allowdisplaybreaks

\begin{document}

\title{\bf Global and Partial Fourier Series for Denjoy-Carleman Classes on Torus}
\author{Alexandre Kirilov \and Bruno de Lessa Victor }
\date{}
\maketitle

\begin{abstract}
In this work, we develop Fourier Analysis for a family of classes of ultradifferentiable functions of Romieu type on the torus, usually known as Denjoy-Carleman classes.  
Then we are able to apply our results in order to generalize some results whose proofs rely almost exclusively on Fourier Series, such as the Greenfield-Wallach Theorem. 
\end{abstract}

\section{Introduction}
\nnfootnote{
\newline{\em 2010 Mathematics Subject Classification. 	35-02, 42A16, 35H10.}
\newline{Key Words and Phrases. Ultradifferentiable Classes; Fourier Series; Global hypoellipticity.}
}

An usual extension of the space of analytic functions, the Gevrey classes have been quite studied in Mathematics, especially in partial differential equations.  Since its origin \cite{gevrey1918nature}, when Maurice Gevrey proved that the Heat Operator is locally $G^{2}$-hypoelliptic but not analytic hypoelliptic, there have been a great number of works concerning solvability and hypoellipticity in the setting of Gevrey functions and ultradistributions (see for instance \cite{baouendi1971etude}, \cite{bove2004gevrey}  and \cite{corli1989local}).

Naturally, mathematicians started to study the same type of problem on compact manifolds, especially the $N$-dimensional torus $\T^N \approx \R^N / 2 \pi \Z^N$ (see \cite{albanese2006global},\cite{cc}, \cite{gramchev1999gevrey}).  It is worth noting that in many related works (\cite{bergamasco2018global}, \cite{himonas2004gevrey}, \cite{arias}), techniques of Fourier Analysis play a fundamental role. 

Nevertheless, despite being very important for the comprehension of the gap between smooth and analytic functions, Gevrey spaces are not enough. Indeed, it is possible to prove that 
$$C^{\omega}(\T^N) \subsetneq \ds \bigcap_{s > 1} \mathcal{G}^{s}(\T^N).$$
Moreover, for each $s >1$ one can construct a non-null flat element of $\mathcal{G}^{s}(\T^N)$, a impossible feat for analytic functions, as well as for a whole family of classes of smooth functions, known as quasianalytic classes. The first example of such class  was given by Denjoy in \cite{denjoy1921fonctions}.  

That leads us to consider another generalization, the \textsl{Denjoy-Carleman Classes}. In few words, given a sequence of positive numbers $\left\{m_n \right\}_{n \in \N_0}$, we study functions $f \in \E(\T^N)$ such that 
$$|D^{\alpha} f(x)| \leq C \cdot h^{|\alpha|} \cdot m_{|\alpha|} \cdot |\alpha|!, \ \ \forall x \in \T^N, \ \forall \alpha \in \N_{0}^N.$$
It is immediate then to ask if results obtained for Gevrey spaces remain valid in this new environment. Since Fourier Series is a  crucial technique for the former case, it seems reasonable the attempt of developing a similar theory for the more general situation. That is the principal aim of the present work, which is intended to facilitate the study of partial differential equations on the torus with setting in classes of ultradifferentiable functions. 

In Section \ref{Star Wars}, we describe precisely our classes of functions, starting by what will be called \textit{weight sequences}. These sequences are absolutely fundamental: not only they characterize the growth of functions and its derivatives, but their conditions also  have a direct impact on the properties of the classes.       

We then endow our classes with a inductive limit topology in Section \ref{The Empire Strikes Back}, which make them DFS spaces. That allows us to define in Section \ref{Return of the Jedi} the spaces of Denjoy-Carleman ultradistributions. Next, we define Global Fourier Series for functions and ultradistributions, characterizing both in terms of the decay of their Fourier Coefficients.

In Section \ref{The Phantom Menace}, we develop Partial Fourier Series, a really important technique for evolution equations. Finally, in Section \ref{Revenge of the Sith} we exhibit some applications such as a extension of Greenfield-Wallach Theorem, an analysis of the Greenfield-Wallach vector field and its connections with Number Theory in our context and the characterization of global hypoellipticity for a class of systems of real vector fields in this framework. 

\section{Denjoy-Carleman Classes on Torus} \label{Star Wars} 

As discussed in Introduction, there exists a strong connection between the chosen sequences and the correspondent spaces. Hence, to define our classes of ultradifferentiable functions properly, we must impose some conditions on $\left\{m_{n} \right\}_{n \in \N_{0}}$.

\begin{Def}
A \textit{weight sequence} is a sequence of positive real numbers $\M = \left\{m_{n} \right\}_{n \in \N_{0}}$ satisfying the following conditions:

\begin{enumerate}[leftmargin=*]
	\item \textit{(Initial Conditions)}
\begin{equation}\label{Rocky}
m_{0} = m_{1} = 1. 
\end{equation}
	\item\textit{(Logarithmic Convexity)}
\begin{equation} \label{Amadeus}
m_{n}^{2} \leq m_{n-1} \cdot m_{n+1}, \ \ \forall n \in \N. 
\end{equation}	
	\item \textit{(Moderate Growth)} There exists $H > 1$ such that  	
\begin{equation} \label{Prisoners}
\ds \sup_{j,k} \left(\ds \frac{m_{j+k}}{m_{j} \cdot m_{k}} \right)^{1/(j+k)} \leq H. 
\end{equation}
\end{enumerate} 
\end{Def}

\begin{Obs}
Condition \eqref{Prisoners}, which we called \textit{"Moderate Growth"} following \cite{kriegl2011convenient} or \cite{thilliez2008quasianalytic},  is also known as \textit{"Stability Under Ultradifferentiable Operators"} in \cite{k1}. Its requirement is due to the fact that we need the \textsl{associated function} (Remark \ref{Apollo 13}) to be a \textsl{weight function} (see \cite{bonet2007comparison}, \cite{braun1990ultradifferentiable}),  which relates two different ways of defining ultradifferentiable functions. 

Technically speaking, \eqref{Prisoners} is equivalent to  Lemma \ref{Jaws}, which in its turn is a crucial tool for the proofs of results such as Theorems \ref{A Clockwork Orange}, \ref{Spider Man II}, \ref{Road to Perdition} and \ref {Butch Cassidy and the Sundance Kid}.
\end{Obs}

\begin{Obs}
Let us analyze the difference between the definition of our sequences and the ones introduced by Komatsu in \cite{k1}. In comparison, we have $m_n \cdot n! = M_n$; the initial conditions \eqref{Rocky} are not demanded in \cite{k1}, but they are essentially imposed to facilitate computations and are not a real obstruction for the theory.

Since \eqref{Prisoners} holds for the factorial sequence, then it is true for $M_n$ if and only if the same happens for $m_n$. On the other hand, 
$$\ds \frac{M_{n-1} \cdot M_{n+1}}{M_n^{2}} = \ds \frac{m_{n+1} \cdot m_{n-1} \cdot (n-1)! \cdot (n+1)!}{m_{n}^{2} \cdot (n!)^{2}} = \ds \frac{m_{n+1} \cdot m_{n-1}}{m_{n}^{2}} \cdot \ds \frac{n+1}{n}. $$
Hence, condition \eqref{Amadeus} is slightly stronger than the $(M.1)$ required by Komatsu. The reason for that will become evident later, in Proposition \ref{Social Network} and its consequences, as well as in Lemma \ref{Casablanca}.
\end{Obs}

\begin{Exe} \label{American Beauty}
Given $s \geq 1$,  $\mathscr{G}_{s} = \left\{(n!)^{s-1}  \right\}_{n \in \N_{0}}$ is a weight sequence. In fact, 

$$\left[\ds \frac{(n+1)! \cdot (n-1)!}{(n!)^2} \right]^{s-1} = \left[\ds \frac{n+1}{n} \right]^{s-1} \geq 1. $$

\noindent and it follows from binomial theorem that 

$$\left[\ds \frac{(j+k)!}{j! \cdot k!} \right]^{s-1} \leq (2^{s-1})^{j+k}. $$

\end{Exe}

\begin{Exe} \label{The Usual Suspects}
Let $\sigma \in \R$, with $\sigma > 0$. Put $\ell_{n} = \left[\log (n + e - 1) \right]^{\sigma \cdot n}$. Let us show that \eqref{Amadeus} holds:

$$\ds \frac{\ell_{n+1} \cdot \ell_{n-1}}{\ell_{n}^{2}} =\left( \ds \frac{[\log(n+e)]^{(n+1)}}{[\log(n+e-1)]^{n}}\cdot \ds \frac{ [\log(n+e-2)]^{(n-1)}}{[\log(n+e-1)]^{n}} \right)^{\sigma}$$

For $n = 1$, it is immediate. Otherwise, we define

$$f:[1, +\infty) \to \R; \ \ f(x) = \ds\frac{[\log(x+e)]^{(x+1)}}{[\log(x+e -1)]^x}.$$

\noindent By analyzing its derivative, we prove that $f$ is nondecreasing and hence $f(n) \geq f(n-1)$, for each $n \geq 2$. Therefore  

$$\ds \frac{\ell_{n+1} \cdot \ell_{n-1}}{\ell_{n}^{2}} = \left(\ds \frac{f(n)}{f(n-1)} \right)^{\sigma} \ \geq 1, \ \ \forall n \geq 2.$$

We proceed to the proof of \eqref{Prisoners}. Given $j, k \in \N$, 
\begin{align*}
\ds \frac{\ell_{j+k}}{\ell_{j} \cdot \ell_{k}} &= \left[ \ds \frac{\log (j + k + e - 1)}{\log (j + e - 1) } \right]^{\sigma j} \cdot \left[ \ds \frac{\log (j + k + e - 1)}{\log (k + e - 1) } \right]^{\sigma k} \\
&\leq \left[1 + \ds \frac{k}{j + e - 1 } \right]^{\sigma j} \cdot \left[1 + \ds \frac{j}{ k + e - 1 } \right]^{\sigma k} \\
&\leq e^{\sigma \cdot (j+k)}.
\end{align*}

\end{Exe}

\begin{Exe} \label{Vertigo}
Consider $m_{n} = \left[\log \left(\log (n + e^{e} - 1) \right) \right]^{\beta \cdot n}$, where $\beta$ is a real positive number. In order to verify \eqref{Amadeus} and \eqref{Prisoners}, the process is very similar. For the proof of former, we define

$$g:[1, +\infty) \to \R; \ \ g(x) = \ds\frac{[\log \left( \log (x+e^{e}) \right)]^{(x+1)}}{[\log \left( \log (x+e^{e} -1) \right)]^{x}}$$

\noindent and show that its derivative is always positive. On the other hand, 
\begin{align*}
\ds \frac{m_{j+k}}{m_{j} \cdot m_{k}} &= \left[ \ds \frac{\log \left(\log (j + k + e^{e} - 1) \right)}{\log \left (\log (j + e^{e} - 1) \right) } \right]^{\beta j} \cdot \left[\ds \frac{\log \left(\log (j + k + e^{e} - 1) \right)}{\log \left(\log (k + e^{e} - 1) \right) } \right]^{\beta k} \\
&\leq \left[ \ds \frac{\log (j + k + e^{e} - 1) }{\log (j + e^{e} - 1)} \right]^{\beta j} \cdot \left[\ds \frac{\log (j + k + e^{e} - 1)}{\log (k + e^{e} - 1)} \right]^{\beta k} \\
&\leq \left[1 + \ds \frac{k}{j + e^{e} - 1 } \right]^{\sigma j} \cdot \left[1 + \ds \frac{j}{ k + e^{e} - 1 } \right]^{\sigma k} \\
&\leq e^{\sigma \cdot (j+k)}.
\end{align*}

\end{Exe}

\begin{Obs}
 Given $\mathscr{L} = \left\{\ell_{k} \right\}_{k \in \N_{0}}$  and $\M = \left\{m_{k} \right\}_{k \in \N_0}$ weight sequences, $\mathscr{N} = \left\{\ell_k \cdot m_k \right\}_{k \in \N_0}$ is also a weight sequence. 
\end{Obs}

\begin{Pro} \label{Social Network}
Let $\M = \left\{m_{n} \right\}_{n \in \N_{0}}$ be a weight sequence. Then the following properties hold: 
\begin{enumerate} [leftmargin=*] 
 \item \label{Schindler's List} $\M$ is non-decreasing. 
 \item \label{Blade Runner} The sequence given by $\beta_{n} := (m_{n})^{ \frac {1}{n}}$ is non-decreasing, for $n \geq 1$. 
 \item \label{On the Waterfront} For every $k, n \in \N_{0}$ such that $k \leq n$, it follows that
$$ m_{k} \cdot m_{n-k} \leq m_{n}.$$ 
 \item \label{The Graduate} Given any $k \in \N$ there exists  $C_{\left\{k\right\}} > 1$ such that: 
\begin{equation} \label{Arrival}
\ds \sup_{j\geq 1} \left(\ds \frac{m_{j+k}}{m_{j}} \right)^{\frac{1}{j}} \leq C_{\left\{k\right\}}.  
\end{equation}
\end{enumerate}
\end{Pro}

\begin{proof}
To prove \ref{Schindler's List}, \ref{Blade Runner} and \ref{On the Waterfront},  we define $\omega_{n} = \log m_{n}$ and apply \eqref{Rocky}, \eqref{Amadeus} to show by induction that 
$$\omega_{j} \leq \ds \frac{j}{j+k} \cdot  \omega_{j+k}, \ \  \forall k \in \N_{0}, \ \forall j \in \N. $$

On the other hand, by \eqref{Prisoners} we have that 
$$ \left(\ds \frac{m_{j+1}}{m_{j}} \right) \leq H^{2j}, \ \ \forall j \in \N_{0}. $$
Therefore,
$$\left(\ds \frac{m_{j+k}}{m_{j}} \right)= \left(\ds \frac{m_{j+1}}{m_{j}} \right) \cdot \left(\ds \frac{m_{j+2}}{m_{j+1}} \right) \cdot \ldots \cdot \left(\ds \frac{m_{j+k}}{m_{j+k-1}} \right) \leq H^{2[j + (j+1) + \ldots + (j+k-1)]} \leq H^{2j \cdot (k^{2} + k)}, \ \ \forall j \geq 1.$$
Thus we complete the proof of \ref{The Graduate} by taking $C_{\left\{k\right\}} = H^{2 \left(k^{2} + k \right)}$.

\end{proof}

\begin{Obs}
Applying Proposition \ref{Social Network}  for the particular case where $s=2$ in the Example \ref{American Beauty}, we obtain for any $k \in \N$ the existence of  $B_{\left\{k \right\}}$ such that

\begin{equation} \label{The Seventh Seal}
\left[\ds \frac{(j+k)!}{j!} \right]^{\frac{1}{j}} \leq B_{\left\{k \right\}}. 
\end{equation}

\end{Obs}

\begin{Def}
Let $\M = \left\{m_{n} \right\}_{n \in \N}$ be a weight sequence. We say that a function $f \in \E(\T^N)$ is \textbf{ultradifferentiable of class $\left\{\M \right\}$} if there exist constants $C, h > 0$ such that, for every $\alpha \in \N_{0}^{N}$,

$$|D^{\alpha}f(x)| \leq C \cdot h^{|\alpha|} \cdot m_{|\alpha|} \cdot  |\alpha|!, \ \ \ \forall x \in \T^N.$$

\end{Def}

\begin{Not}
We denote the space of ultradifferentiable functions of class $\left\{\M \right\}$ defined above as $\E_{\M}(\T^N)$. We call $\E_{\M}(\T^N)$ the \textit{Denjoy-Carleman Class on $\T^{N}$ associated to $\M$.}  
\end{Not}

It is proved in \cite{thilliez2008quasianalytic}, using a construction of Bang in \cite{bang1946om}, that for each weight sequence $\M$ the set $\E_{\M}(\T^N)$ is not trivial. It is very easy to check that it is also a vector subspace of $\E(\T^N)$. Moreover, each statement proved in the Proposition \ref{Social Network} plays a fundamental role in properties of the Denjoy-Carleman Classes associated. For instance, the fact that a weight sequence is non-decreasing (\ref{Schindler's List}) tell us that $m_n \geq 1$ and therefore
$$C^{\omega}(\T^N) \subset \E_{\M}(\T^N). $$ 

It follows from \ref{Blade Runner} that $\E_{\M}(\T^N)$ is always closed for composition and inverse-closed (see \cite{bierstone2004resolution}, \cite{rudin1962division}). By \ref{On the Waterfront} and \ref{The Graduate}, it is not difficult to show that the classes are a subalgebra of $\E(\T^N)$ and closed by differentiation, respectively.

Another important subject related to Denjoy-Carleman classes, the one which actually originated its study, is the existence of flat functions.  

\begin{Def}
Let $\M$ be a weight sequence and $\E_{\M}(\T^N)$ its class of ultradifferentiable functions associated. We say that $\E_{\M}(\T^N)$ is \textbf{quasianalytic} if

$$f \in \E_{\M}(\T^N), \ x_{0} \in \T^N  \ \big{|} \ D^{\alpha}f (x_{0}) = 0, \ \forall \alpha \in \N_{0}^N \ \Rightarrow f \equiv 0.$$
\end{Def}

The characterization for quasianalyticity was first given partially by Denjoy (\cite{denjoy1921fonctions}) and then completely by Carleman (\cite{carleman1926fonctions}). We recommend \cite{hormander2003analysis}, \cite{rudin2006real} for a proof. 

\begin{Teo}
Let $\M = \left\{m_{n} \right\}_{n \in \N_{0}}$ be a weight sequence. The Denjoy-Carleman class $\E_{\M}(\T^N)$ associated is quasianalytic if and only if 

$$\ds \sum_{j = 0}^{+ \infty} \ds \frac{m_j}{m_{j+1} \cdot (j+1)} = + \infty. $$
\end{Teo}

Let us discuss a little about the examples set before. The classes associated to the sequences defined in the Example \ref{American Beauty} are the famous \textbf{Gevrey Classes}, which are usually denoted as $\mathcal{G}^{s}(\T^N)$. When $s = 1$, we obtain the class of analytic functions, that is obviously a quasianalytic class. If $s > 1$, the classes are non-quasianalytic. 

The classes associated to the Examples \ref{The Usual Suspects} are very important for comprehension of the gap between smooth functions and analytic functions, since the particular case where $\sigma = 1$ it represents the intersection of all inverse-closed non-quasianalytic classes (\cite{rudin1962division}). Besides, it is quasianalytic for $0 < \sigma \leq 1$ and non-quasianalytic for $1 < \sigma$ (\cite{thilliez2008quasianalytic}).

The Example \ref{Vertigo} is very relevant historically, since the particular case where $\beta = 1$ was introduced by Denjoy in \cite{denjoy1921fonctions} as a first example of class of quasianalytic functions which were nowhere analytic. 

Before the end of the section, we ask the following question: given $\mathscr{L}, \M$ weight sequences, when $\E_{\mathscr{L}} (\T^N) \subset \E_{\M}(\T^N)$? The answer is onece again in assymptotic behavior of the sequences.  

\begin{Def}
Let $\mathscr{L}, \M$ be weight sequences. Then 
\begin{equation} \label{The Godfather}
\mathscr{L} \preceq \M \ \Leftrightarrow \ds \sup_{k \in \N_{0}} \left(\ds \frac{\ell_k}{m_k} \right)^{\frac{1}{k}} < + \infty.  
\end{equation}

\noindent Moreover, we write $\mathscr{L} \approx \M$ when $\mathscr{L} \preceq \M$ and $\M \preceq \mathscr{L}$.  
\end{Def}

It is possible to prove (see \cite{thilliez2008quasianalytic}) that 
\begin{equation} \label{The Godfather II}
\E_{\mathscr{L}}(\T^N) \subset \E_{\M} (\T^N) \Leftrightarrow \mathscr{L} \preceq \M.  
\end{equation}

\noindent In particular, we have $\E_{\mathscr{L}}(\T^N) = \E_{\M}(\T^N) = $ if and only if $\mathscr{L} \approx \M$. This shows, for instance,  that $r < s$ implies thar $\mathcal{G}^{r} (\T^N) \subsetneq \mathcal{G}^{s} (\T^N)$ and  $C^{\omega}(\T^N) = \E_{\M}(\T^N)$ if and only if $\ds \sup_{j\geq 1} \ (m_j)^{\frac{1}{j}} < +\infty$.

\begin{Obs}
From now on, we are going to fix some  arbitrary weight sequence $\M$ and to develop our whole theory based on it.  
\end{Obs}

\section[Topology of Ultradifferentiable Functions]{The topology of  $\E_{\M}(\T^N)$} \label{The Empire Strikes Back}

\begin{Def} \label{Ocean's Eleven}
Given $h> 0$, we define

$$\E_{\M, h}(\T^N) = \left\{f \in \E_{\M}(\T^N); \ \underset{\alpha \in \N_{0}^{N}}{\ds \sup_{{x \in \T^N}}} \left| \ds \frac{D^{\alpha} f(x)}{h^{|\alpha|} \cdot m_{|\alpha|} \cdot |\alpha|!}  \right|. < \infty \right\}.$$ 
\end{Def}

\noindent Moreover, we denote for $f \in \E_{\M, h}(\T^N)$

$$\left\|f \right\|_{\M, h}:= \underset{\alpha \in \N_{0}^{N}}{\ds \sup_{{x \in \T^N}}} \left| \ds \frac{D^{\alpha} f(x)}{h^{|\alpha|} \cdot m_{|\alpha|} \cdot |\alpha|!}  \right|.$$

One can easily show that $\left\|. \right\|_{\M, h}$ is a norm in $\E_{\M, h}(\T^N)$ and that in this case it is a Banach space.

\begin{Pro} \label{Adaptation}
Let $h_{1}, h_{2} \in  \R_{+}$, with $h_{1} < h_{2}$. Thus $\E_{\M, h_{1}}(\T^N) \subset \E_{\M, h_{2}}(\T^N)$ and the inclusion $\E_{\M, h_{1}}(\T^N) \hookrightarrow \E_{\M, h_{2}}(\T^N)$ is compact.
\end{Pro}

\begin{proof}
We are only going to prove the compactness: let $\left\{f_{n} \right\}_{n \in \N}$ be a bounded sequence in $\E_{\M, h_{1}}(\T^N)$. Hence, there exists $C_{0} > 0$ such that:
$$|D^{\alpha}f_{n}(x)| \leq C_{0} \cdot h_{1}^{|\alpha|} \cdot m_{|\alpha|} \cdot |\alpha|!, \ \ \forall x \in \T^N, \ \forall \alpha \in \N_{0}^{N}, \ \ \forall n \in \N. $$
Thus for every $k \in \N_{0}$ there exists $C_{k} > 0$ such that

$$ \ds \sum_{|\beta| \leq k} \ds \sup_{x \in \T^N} |D^{\beta} f_{n}(x)| \leq C_{k}, \ \forall n \in \N.$$

It follows from Arzel\`a-Ascoli Theorem the existence of a subsequence $\left\{f_{n_{k}} \right\}_{k \in \N}$ that converges to $f$ in $\E(\T^N)$. Moreover, for any $x \in \T^N$, $\gamma \in \N_{0}^{N}$, 

$$|D^{\gamma}f(x)| = \left|\ds \lim_{k \to + \infty} D^{\gamma}f_{n_{k}}(x) \right| \leq C_{0} \cdot h_{2}^{|\gamma|} \cdot m_{|\gamma|} \cdot |\gamma|!,$$

\noindent Thus $f \in \E_{\M, h_{2}}(\T^N)$. It leaves for us to prove that $f_{n_{k}} \to f$ in $\E_{\M, h_{2}}(\T^N)$.

Given $\varepsilon > 0$, we take $p \in \N$ such that $\left(\ds \frac{h_{1}}{h_{2}} \right)^{p} \leq \ds \frac{\varepsilon}{2C_{0}}$ and $C_{1}:= \ds \max \left\{\ds \frac{1}{h_{2}^{q} \cdot m_{q} \cdot q!}; \  0 \leq q \leq p \right\}.$ Due to the fact that $f_{n_{k}} \to f$ in $\E(\T^N)$, it is possible to find $k_{1} \in \N$ such that

$$\ds \sup_{x \in \T^N} \left|D^{\lambda}f_{n_{k}}(x) - D^{\lambda}f(x) \right| \leq \ds \frac{\varepsilon}{C_{1}}, \ \  |\lambda| \leq p, \ \  k \geq k_{1}.$$  

Let $k \geq k_{1}$; if $|\lambda| \leq p$,   

$$\ds \sup_{x \in \T^N} \ds \frac{\left|D^{\lambda}f_{n_{k}}(x) - D^{\lambda}f(x) \right|}{h_{2}^{|\lambda|} \cdot m_{|\lambda|} \cdot |\lambda|!} \leq \ds \frac{\varepsilon}{C_{1}} \cdot C_{1} = \varepsilon. $$

\noindent When $|\lambda| > p,$

\begin{align*}
\ds \sup_{x \in \T^N} \ds \frac{\left|D^{\lambda}f_{n_{k}}(x) - D^{\lambda}f(x) \right|}{h_{2}^{|\lambda|} \cdot m_{|\lambda|} \cdot |\lambda|!} &\leq \left[\ds \sup_{x \in \T^N} \ds \frac{\left|D^{\lambda}f_{n_{k}}(x) - D^{\lambda}f(x) \right|}{h_{1}^{|\lambda|} \cdot m_{|\lambda|} \cdot |\lambda|!}\right] \cdot \left(\ds \frac{h_{1}}{h_{2}} \right)^{p} \\
&\leq \left[\left\|f_{n_{k}} \right\|_{\M, h_{1}} + \left\|f \right\|_{\M, h_{1}} \right] \cdot \ds \frac{\varepsilon}{2C_{0}} \\
&\leq \varepsilon. 
\end{align*}

\noindent Therefore $\left\|f_{n_{k}} - f \right\|_{\M, h_{2}} \leq \varepsilon$ for all $k \geq k_{1}$, which shows that $f_{n_{k}} \to f$ in $\E_{\M, h_{2}}(\T^N)$. 
\end{proof}

Let $\left\{h_{n} \right\}_{n \in \N}$ be a strictly increasing sequence of positive real numbers, such that $h_{n} \to + \infty$. It is not difficult to see that 
$$\E_{\M}(\T^N) = \bigcup_{n \in \N} \E_{\M, h_{n}}(\T^N).$$
We endow $\E_{\M}(\T^N)$ with the inductive limit topology given by the family of $\E_{\M, h_{n}}(\T^N)$. That is, 

$$\E_{\M}(\T^N) = \ds \lim_{\stackrel{\longrightarrow}{n \in \N}} \E_{\M, h_{n}}(\T^N). $$

It follows from Proposition \ref{Adaptation} that $\E_{\M}(\T^N)$ is a \textit{injective limit of a compact sequence  of locally convex spaces}, also known in the literature as a \textbf{DFS} space.  For more details, see \cite{k2}.

\begin{Obs}
The topology introduced above does not depend on the choice of $\left\{h_{n} \right\}_{n \in \N}$.
\end{Obs}

\section{Ultradistributions and Fourier Series} \label{Return of the Jedi}

\begin{Def} 
We define $\D'_{\M}(\T^N)$ as the \textbf{topological dual space of $\E_{\M}(\T^N)$}. That is,  the space of continuous linear functionals $u:\E_{\M}(\T^N) \to \C$ .
\end{Def}

\begin{Teo} \label{The Hurt Locker} 
Let $u: \E_{\M}(\T^N) \to \C$ be a linear functional. The following properties are equivalent:

\begin{enumerate} [leftmargin=*]
	\item \label{Scarface} $u\in \D'_{\M}(\T^N)$. 
	\item \label {Mission Impossible} For every $\varepsilon >0$, there exists $C_\varepsilon > 0$ such that
	
\begin{equation} \label{Minority Report}	
\left|\<u, \varphi \> \right| \leq C_\varepsilon \cdot \underset{\alpha \in \N_{0}^{N}}{\ds \sup_{{x \in \T^N}}} \left( \ds \frac{|\partial^\alpha \varphi (x)| \cdot \varepsilon^{|\alpha|}} {m_{|\alpha|} \cdot |\alpha|!} \right), \ \ \ \ \ \forall \varphi \in \E_\M(\T^N).
\end{equation}

  \item \label{The Untouchables} If $\left\{\varphi_n \right\}_{n \in \N} \subset \E_{\M}(\T^N)$ converges to $0$ in $\E_{\M}(\T^N)$, then $\<u, \varphi_{n} \> \to 0. $
\end{enumerate}
\end{Teo}

\begin{proof}

\noindent $\eqref{Scarface} \Rightarrow \eqref{Mission Impossible}:$ If $\eqref{Mission Impossible}$ does not hold, there exists $\varepsilon_0 >0$ and $\left\{\varphi_{n} \right\}_{n \in \N} \subset \E_{\M}(\T^N)$ such that
$$\left|\<u, \varphi_{n} \> \right| > n \cdot \underset{\alpha \in \N_{0}^{N}}{\ds \sup_{{x \in \T^N}}} \left( \ds \frac{|\partial^\alpha \varphi_{n} (x)| \cdot \varepsilon_{0}^{|\alpha|}} {m_{|\alpha|} \cdot |\alpha|!} \right).$$
By putting $\Psi_n = \ds \frac{\varphi_{n}}{| \<u, \varphi_{n} \>|}$, we get that 
$$1 >  n \cdot \underset{\alpha \in \N_{0}^{N}}{\ds \sup_{{x \in \T^N}}}  \left(\ds \frac{|\partial^\alpha \Psi_{n} (x)| \cdot \varepsilon_{0}^{|\alpha|}} {m_{|\alpha|} \cdot |\alpha|!} \right)  \Rightarrow \underset{\alpha \in \N_{0}^{N}}{\ds \sup_{{x \in \T^N}}} \left(  \ds \frac{|\partial^\alpha \Psi_{n} (x)| \cdot \varepsilon_{0}^{|\alpha|}} {m_{|\alpha|} \cdot |\alpha|!} \right) < \ds \frac{1}{n}. $$

Hence

$$|\partial^{\alpha} \Psi_{n}(x)| \leq \left(\ds \frac{1}{\varepsilon_{0}} \right)^{|\alpha|} \cdot m_{|\alpha|} \cdot |\alpha|!, \ \ \ \ \forall x \in \T^N, \ \  \forall \alpha \in \N_{0}^{N},$$

\noindent which allows us to deduce that $\left\{\Psi_{n} \right\}_{n \in \N} \subset \E_{\M, {1/\varepsilon_{0}}}(\T^N)$. Thus 

$$|\<u, \Psi_n \>| \geq n \left\|\Psi_{n} \right\|_{\M, 1/\varepsilon_{0}}, \ \ \ \forall n \in \N, $$

\noindent and $u \Big{|}_{\E_{\M, 1/\varepsilon_{0}}(\T^N)}$ is not continuous. So $u$ is not continuous.  

\

\noindent $\eqref{Mission Impossible} \Rightarrow \eqref{The Untouchables}:$ Suppose $\left\{\varphi_n \right\}_{n \in \N} \subset \E_{\M}(\T^N)$ a sequence that converges to $0$ in the same space. By a property of DFS spaces,  there exists $p \in \N$ such that $(\varphi_n)_{n \in \N} \subset \E_{\M, h_p}(\T^N)$ and $ (\varphi_n)_{n \in \N} \to 0$ in  $\E_{\M, h_p}(\T^N)$. If we define $\varepsilon = \ds \frac{1}{h_p}$, by hypothesis one can find $C>0$ satisfying

$$|\<u, \varphi_n \>|  \leq C \cdot \underset{\alpha \in \N_{0}^{N}}{\ds \sup_{{x \in \T^N}}} \left(  \ds \frac{|\partial^\alpha \varphi_n (x)|} {(h_p)^{|\alpha|} \cdot m_{|\alpha|} \cdot |\alpha|!} \right) \leq C \cdot \left\| \varphi_n \right\|_{\M, h_p}.$$

\noindent Since $\left\| \varphi_n \right\|_{\M, h_p} \to 0$, we have that $|\<u, \varphi_n \>| \to 0$.

\

\noindent $\eqref{The Untouchables} \Rightarrow \eqref{Scarface}:$ Assume that $u \notin \D'_{\M}(\T^N)$; then there exists $q \in \N$ such that $u \big{|}_{\E_{\M, h_q}(\T^N)}$ is not continuous. Thus for every $j \in \N$ we obtain $\varphi_j \in \E_{\M, h_q}(\T^N)$

$$|\<u, \varphi_j \>| > j \cdot \left\| \varphi_j \right\|_{\M, h_q}.$$

Let $\Psi_j := \ds \frac{\varphi_j}{|\<u, \varphi_j \>|}$; then 
$$| \<u, \Psi_j \>| = 1,  \ \ \ \left\| \Psi_j \right\|_{\M, h_q} = \ds \frac{\left\| \varphi_j \right\|_{\M, h_q}}{|\<u, \varphi_j \>|} < \ds \frac{1}{j}, \ \ \ \ \forall j \in \N. $$
Hence, $\left\{\Psi_j \right\}_{j \in \N} \to 0$ in $\E_{\M}(\T^N)$, whereas $\<u, \Psi_j \> \nrightarrow 0$. 
\end{proof}

\

\begin{Def}
Let $\varphi \in \E_{\M}(\T^N)$; for each $\xi \in \Z^N$ we define its Fourier coefficient as:

$$\hat{\varphi}(\xi):= \ds \frac{1}{(2 \pi)^N} \ds \int_{\T^N} e^{-ix\xi} \cdot \varphi(x) dx.$$

\end{Def}

\begin{Teo} \label{Kramer vs Kramer}
For every $\varphi \in \E_{\M}(\T^N)$,

$$\varphi(x) = \ds \sum_{\xi \in \Z^N} \hat{\varphi}(\xi) \cdot e^{ix \xi},  \ \ \ \ \forall x \in \T^N,$$

\noindent with convergence in $\E_{\M}(\T^N)$. Moreover, there exist constants $C, \delta > 0$ such that

\begin{equation} \label{Four Weddings and a Funeral}
|\hat{\varphi}(\xi)| \leq C \cdot \inf_{n \in \N_{0}} \left( \ds \frac {m_n \cdot n!} {\delta^{n} \cdot (1+|\xi|)^{n}} \right), \ \ \ \ \forall \xi \in \Z^N. 
\end{equation}
\end{Teo}

\begin{proof}
Let us prove $\eqref{Four Weddings and a Funeral}$ first; when $\xi = 0$ the estimate is obvious for $\delta = 1$, so we may consider $\xi \neq 0$. It is easy to the see that

\begin{equation} \label{Raiders of the Lost Ark}
|\hat{\varphi}(\xi)| \leq \sup_{x \in \T^N} |\varphi(x)|.
\end{equation}
For any $\alpha \in \N^{N}_{0}$ different of zero, we have $\widehat{D^{\alpha} \varphi}(\xi) = \xi^{\alpha} \cdot \widehat{\varphi}(\xi). $ Hence, there exist $C_1, h_1 > 1$ such that
$$|\xi^\alpha| \cdot |\hat{\varphi}(\xi)| \leq C_1 \cdot h_{1}^{|\alpha|} \cdot m_{|\alpha|} \cdot |\alpha|!.$$

Since $|\xi|^{|\alpha|} \leq \ds \sum_{|\beta| = |\alpha|} \ds \frac{|\alpha|!}{\beta !} \cdot |\xi^\beta|$, 
\begin{align*}
|\xi|^{|\alpha|} \cdot |\hat{\varphi}(\xi)| &\leq \ds \sum_{|\beta| = |\alpha|} \ds \frac{|\alpha|!}{\beta !} \cdot |\xi^{\beta}| \cdot |\hat{\varphi}(\xi)| \\
& \leq C_1 \cdot h_{1}^{|\alpha|} \cdot m_{|\alpha|} \cdot |\alpha|! \cdot  \ds \sum_{|\beta| = |\alpha|} \ds \frac{|\alpha|!}{\beta !} \\
& \leq C_1 \cdot (h_{1} \cdot N)^{|\alpha|} \cdot m_{|\alpha|} \cdot |\alpha|!. 
\end{align*}
Because $\xi \neq 0$, we infer that

\begin{equation} \label{Temple of Doom}
(1+ |\xi|)^{|\alpha|}  \cdot |\hat{\varphi}(\xi)|  \leq C_1 \cdot (2 \cdot h_{1} \cdot N)^{|\alpha|} \cdot m_{|\alpha|} \cdot |\alpha|!. 
\end{equation}

Considering that  $\alpha$ is arbitrary, by taking $C = \max\left\{\ds \sup_{x \in \T^N} |\varphi(x)|, C_{1} \right\}$ and $\delta = \ds \frac{1}{2 \cdot h_{1} \cdot N}$, it follows from \eqref{Raiders of the Lost Ark} and \eqref{Temple of Doom} that
$$|\hat{\varphi}(\xi)|\leq C \cdot \ds \inf_{n \in \N_{0},} \left(\ds \frac{m_{n} \cdot n!}{\delta^n \cdot (1+|\xi|)^{n}} \right), \ \  \forall \xi \in \Z^N.$$

We proceed to the convergence of the series.  Since $\varphi \in \E(\T^N)$, the result holds for the same space. For each $k \in \N$, consider

$$S_{k} \varphi(x) = \ds \sum_{|\xi| \leq k} \hat{\varphi}(\xi) \cdot e^{i \xi x}$$

\noindent $S_k$ is an analytic function and consequently an  element of $\E_{\M}(\T^N)$. Besides, given $\alpha \in \N_{0}^{N}$,

$$D^{\alpha}(\varphi - S_{k} \varphi)(x) = \ds \sum_{|\xi| \geq k+1} \hat{\varphi}(\xi) \cdot \xi^\alpha \cdot  e^{i \xi x}.$$

\noindent Thereafter
\begin{align*}
\left|D^{\alpha}(\varphi - S_{k} \varphi)(x) \right| &\leq \ds \sum_{|\xi| \geq k+1} \left|\hat{\varphi}(\xi) \right| \cdot \left|\xi \right|^{|\alpha|} \\
& \leq C \cdot \left(\ds \frac{1}{\delta} \right)^{|\alpha| + 2N} \cdot m_{|\alpha| + 2N} \cdot (|\alpha| + 2N)! \cdot  \ds \sum_{|\xi| \geq k+1}  \left(1+ |\xi|\right)^{-2N}.
\end{align*}

From \eqref{Arrival} and \eqref{The Seventh Seal}, 
\begin{align*}
\left|D^{\alpha}(\varphi - S_{k} \varphi)(x) \right| &\leq C \cdot \left(\ds \frac{1}{\delta} \right)^{|\alpha| + 2N} \cdot (m_{|\alpha|} \cdot C_{\left\{2N \right\}}^{|\alpha|}) \cdot (|\alpha|! \cdot B_{\left\{2N \right\}}^{|\alpha|}) \cdot  \ds \sum_{|\xi| \geq k+1}  \left(1+ |\xi|\right)^{-2N} \\
& \leq \left[C \cdot \left(\ds \frac{1}{\delta} \right)^{2N} \cdot \ds \sum_{\xi \in \Z^N}  \left(1+ |\xi|\right)^{-2N}\right] \cdot \left(\left(\ds \frac{1}{\delta} \right) \cdot B_{\left\{2N \right\}} \cdot C_{\left\{2N \right\}} \right)^{|\alpha|} \cdot m_{|\alpha|} \cdot |\alpha|!
\end{align*}

\noindent By defining $C' = C \cdot \left(\ds \frac{1}{\delta} \right)^{2N} \cdot \ds \sum_{\xi \in \Z^N}  \left(1+ |\xi|\right)^{-2N}$ and $h_{1}'= \left(\ds \frac{1}{\delta} \right) \cdot B_{\left\{2N \right\}} \cdot C_{\left\{2N \right\}}$, we obtain that  

$$\left|D^{\alpha}(\varphi - S_{k} \varphi)(x) \right| \leq C' \cdot h_{1}'^{|\alpha|} \cdot m_{|\alpha|} \cdot |\alpha|!, \ \ \ \forall x \in \T^N, \ \forall \alpha \in \N_{0}^{N}. $$

Since $h_{1} < h_{1}'$, $\varphi$ and  $(\varphi - S_{k} \varphi)$ belong to $\E_{\M, h_{1}'}(\T^N)$,  which shows us that the same is true for $S_{k}$, for each $k \in \N$. In addition,  
$$\left|D^{\alpha}(\varphi - S_{k} \varphi)(x) \right| \leq \left[C \cdot \left(\ds \frac{1}{\delta} \right)^{2N} \cdot \ds \sum_{|\xi| \geq k+1}  \left(1+ |\xi|\right)^{-2N}\right] \cdot (h_{1}')^{|\alpha|} \cdot m_{|\alpha|} \cdot |\alpha|!, \ \ \ \forall x \in \T^N, \ \forall \alpha \in \N_{0}^{N}.$$

\noindent Therefore $\ds \lim_{k \to + \infty} \left\|\varphi - S_{k} \cdot \varphi \right\|_{\M, h_1'} =0$, which implies that $S_{k}\varphi \to \varphi$ in $\E_{\M}(\T^N)$. 
\end{proof}

We extend now the notion of Fourier Series for ultradistributions.

\begin{Def}
Let $u \in \D_{\M}^{'}(\T^N)$; we define its Fourier Coefficient  $\hat{u}(\xi)$ as   

$$\hat{u}(\xi) = \ds \frac{1}{(2 \pi)^N} \cdot \<u, e^{-ix\xi} \>, \ \ \forall \xi \in \Z^{\N}. $$
\end{Def}

\begin{Teo} \label{Dunkirk}
Let $u \in \D_{\M}'(\T^N)$. For every $\varepsilon >0$, there exists $C_{\varepsilon} > 0$ such that
$$|\hat{u}(\xi)| \leq C_{\varepsilon} \cdot \ds \sup_{n \in \N_{0}} \left( \ds \frac{ \varepsilon^{n} \cdot (1 +|\xi|)^{n} } {m_{n} \cdot n!} \right), \ \ \forall \xi \in \Z^N. $$

\end{Teo}

\begin{proof}
For every $\varepsilon > 0, \ \xi \in \Z^N$ it follows from \eqref{Minority Report} the existence of $C_{\varepsilon} >0$ satisfying
$$ |\hat{u}(\xi)| \leq C_{\varepsilon} \cdot \underset{\alpha \in \N_{0}^{N}}{\ds \sup_{{x \in \T^N}}} \left( \ds \frac{|\partial_{x}^\alpha (e^{-i x \xi}) | \cdot \varepsilon^{|\alpha|}} {m_{|\alpha|} \cdot |\alpha|!} \right).$$
Consequently, 
$$|\hat{u}(\xi)| \leq C_{\varepsilon} \cdot {\ds \sup_{\alpha \in \N_{0}^{N}}} \left( \ds \frac{(1 +|\xi^{\alpha}|) \cdot \varepsilon^{|\alpha|}} {m_{|\alpha|} \cdot |\alpha|!} \right) \leq C_{\varepsilon} \cdot \sup_{\alpha \in \N_{0}^{N}} \left(\ds  \ds \frac{(1 +|\xi|)^{|\alpha|} \cdot \varepsilon^{|\alpha|}} {m_{|\alpha|} \cdot |\alpha|!} \right),$$
which allows us to infer that

$$|\hat{u}(\xi)| \leq C_{\varepsilon} \cdot \ds \sup_{n  \in \N_{0}} \left( \ds \frac{ \varepsilon^{n} \cdot (1 +|\xi|)^{n} } {m_{n} \cdot n!} \right).$$
\end{proof}

\begin{Obs} \label{Apollo 13} 
For any  $t > 0$, it is easy to see that $\ds \lim_{n \to + \infty} \left(\ds \frac{t^{n}}{m_{n} \cdot n!} \right) = 0.$ Besides, 

$$\ds \sup_{n \in \N_{0}} \left(\ds \frac{t^{n}}{m_{n} \cdot n!} \right) \geq  \left(\ds \frac{t^{0}}{m_{0} \cdot 0!} \right) = 1.$$

Therefore $\ds \sup_{n \in \N_{0}} \left(\ds \frac{t^{n}}{m_{n} \cdot n!} \right)$ is always assumed by some $n_{0} \in \N_{0}$. A similar argument shows that the same is valid for $\ds \inf_{n \in \N_{0}} \left( \ds \frac{m_{n} \cdot n!}{t^{n}} \right)$. Hence

$$\left[\ds \sup_{n \in \N_{0}} \left(\ds \frac{t^{n}}{m_{n} \cdot n!} \right)\right]^{-1} = \ds \inf_{n \in \N_{0}} \left( \ds \frac{m_{n} \cdot n!}{t^{n}} \right) .$$

In the Gevrey case for instance, the Theorems \ref{Kramer vs Kramer} and \ref{Dunkirk} may be rewritten in the following way: if $u$ is a $s$-Gevrey ultradistribution, then

$$\forall \varepsilon, \exists \ C_{\varepsilon} > 0; \ |\widehat{u}(\xi)| \leq C_{\varepsilon} \cdot e^{\varepsilon \cdot |\xi|^{1/s}}, \ \ \forall \xi \in \Z^{N}.$$

\noindent Moreover, if $\varphi$ is a $s$-Gevrey function, 

$$\exists \ C, \delta > 0; \ |\widehat{\varphi} (\xi)| \leq C \cdot e^{-\delta \cdot |\xi|^{1/s}}, \ \ \forall \xi \in \Z^N. $$

This is due to the fact that the functions $\ds \sup_{n \in \N_{0}} \left(\ds \frac{t^{n}}{n!^{s}} \right)$ and $e^{t^{1/s}}$ are in a certain way equivalent. That is, one can show that both characterizations of ultradistributions and functions are equivalent. 

For a full description of properties of the function $t \mapsto \ds \sup_{n \in \N_{0}} \left(\ds \frac{t^{n}}{m_{n} \cdot n!} \right)$ we recommend the section \textit{Associated Functions} in \cite{k1}. Here, we just state a result that will be important later:  
\end{Obs}

\begin{Lem} \label{Jaws} (Proposition 3.6 of \cite{k1}) 
Let $H$ as in \eqref{Prisoners}; then 

$$ \left[\ds \sup_{n \in \N_{0}} \left(\ds \frac{\rho^n}{m_n \cdot n!} \right) \right]^2 \leq \ds \sup_{n \in \N_{0}} \left(\ds \frac{\rho^n \cdot H^n}{m_n \cdot n!}  \right), \ \ \forall \rho > 0. $$
 
\end{Lem}

\

The next step is to prove a version of the Theorem $\ref{Kramer vs Kramer}$ for ultradistributions.   

\begin{Def}
Let $\left\{u_n\right\}_{n \in \N}$ be a sequence in $\D'_{\M}(\T^N)$ and $u \in \D'_{\M}(\T^N)$. We say that $u_n \to u$ if
$$\<u_n, \varphi \> \to \<u, \varphi \>, \ \ \forall \varphi \in \E_{\M}(\T^N).$$
\end{Def}

\begin{Lem} \label{Magnolia}
Let $\left\{u_n\right\}_{n \in \N}$ be a sequence in $\D'_{\M}(\T^N)$ such that $\<u_n, \varphi \>$ is a Cauchy sequence in $\C$ for every $\varphi \in \E_{\M}(\T^N)$. Then there exists $u \in \D'_{\M}(\T^N)$ such

$$\lim \< u_n, \varphi \> = \<u, \varphi \>, \ \ \forall \varphi \in \E_{\M}(\T^N).$$
\end{Lem}

\begin{proof}
Let $u: \E_{\M}(\T^N) \to \C; \ \ \<u, \varphi \> = \lim \<u_{n}, \varphi \>$. We state that $u$ is continuous; in fact, let  $\left\{\varphi_{k} \right\}_{k \in \N}$ be a sequence in $\E_{\M}(\T^N)$ converging to $0$. Then there exists $q \in \N$ such that $\varphi_{k} \in \E_{\M, h_{q}}(\T^N)$ for every  $k \in \N$ and $\left\|\varphi_{k} \right\|_{\M, h_{q}} \to 0$. 

Given $\psi \in \E_{\M, h_q}(\T^N)$, the sequence $\<u_n, \psi\>$ is bounded. Since $u_n \big{|}_{\E_{\M, h_q}(\T^N)}$ is continuous for each $n \in \N$, it follows from Uniform Boundedness Principle that 

$$|\<u_n, \psi \>| \leq M \cdot \left\|\psi \right\|_{\M, h_q}, \ \ \forall n \in \N, \ \forall \psi \in \E_{\M, h_q}(\T^N),$$

\noindent for some $M > 0$. We then fix $\varepsilon >0$ and define $\gamma_{k} := \ds \frac{2M}{\varepsilon} \cdot \varphi_{k}$, $\forall k \in \N$. It is immediate that $\left\{\gamma_{k} \right\}_{k \in \N} \subset \E_{\M, h_q}(\T^N)$ and $\gamma_{k} \to 0$ in $\E_{\M, h_q}(\T^N)$.

By choosing  $k_{0} \in \N$ such that $\left\|\gamma_{k} \right\|_{\M, h_q} \leq 1$ for every $ k \geq k_0$, we obtain 

$$| \<u_{n}, \gamma_{k} \>| \leq M \ \ \Rightarrow | \<u_{n}, \varphi_{k} \>| \leq \ds \frac{\varepsilon}{2} , \ \ \forall n \in \N, \ \forall k \geq k_0.$$

\noindent Because $\<u, \varphi_{k} \> = \ds \lim_{n} \<u_n, \varphi_{k} \>$ for each $k \geq k_0$ we can find $n_{k} \in \N$ satisfying

$$\left|\<u, \varphi_{k} \> - \< u_{n_{k}}, \varphi_{k} \> \right| \leq \ds \frac{\varepsilon}{2}. $$

Therefore we have for $k \geq k_0$:
\begin{align*}
|\<u, \varphi_{k} \>| &\leq \left|\<u, \varphi_{k} \> - \<u_{n_{k}}, \varphi_{k} \> \right| + \left| \<u_{n_{k}}, \varphi_{k} \> \right| \leq \varepsilon, 
\end{align*} 
which proves our assertion, by Theorem \ref{The Hurt Locker}. 
\end{proof}

\begin{Teo} \label{Memento}
Let $\left\{a_{\xi} \right\}_{\xi \in \Z^{N}}$ be a sequence of complex numbers satisfying the following condition: for every $\varepsilon >0$ there exists $C_{\varepsilon} >0$ such that

$$|a_{\xi}| \leq  C_{\varepsilon} \cdot \ds \sup_{n \in \N_{0}} \left( \ds \frac{\varepsilon^{n} \cdot (1+ |\xi|)^{n}}{m_{n} \cdot n!} \right), \ \ \forall \xi \in \Z^N.$$

\noindent Then $u = \ds \sum_{\xi \in \Z^{N}} a_{\xi} \cdot e^{ix \xi}$ belongs to $\D_{\M}'(\T^N)$ and $\hat{u}(\xi) = a_{\xi}$.
\end{Teo}

\begin{proof}
Let $s_j(x) = \ds \sum_{|\xi| \leq j} a_{\xi} \cdot e^{i x \xi}$, for every $j \in \N$. We fix $\varphi \in \E_{\M}(\T^N)$ and claim that 

$$\<s_j, \varphi\>:= \ds \int_{\T^N} \ds \sum_{|\xi| \leq j} a_{\xi} \cdot e^{i x \xi} \cdot \varphi(x) dx = \ds \sum_{|\xi| \leq j} a_{\xi} \cdot \ds \int_{\T^N} e^{i x \xi} \cdot \varphi(x) dx$$

\noindent is a Cauchy sequence in $\C$. In fact, by taking $m, k \in \N$ with $m >k$,  

$$\<s_m - s_k, \varphi \> = (2 \pi)^{N} \cdot \ds \sum_{k+1 \leq |\xi| \leq m} a_{\xi} \cdot \hat{\varphi}(-\xi).$$

It follows from Theorem \ref{Kramer vs Kramer} the existence of  $C_{1}, \delta > 0$ such that

$$ |\hat{\varphi}(\xi)| \leq C_{1} \cdot \inf_{n \in \N_{0}}\left(\ds \frac{ m_n \cdot n!}{\delta^{n} \cdot 
(1+ |\xi|)^{n}} \right), \ \ \ \forall \xi \in \Z^N. $$

\noindent Then, for some $\varepsilon > 0$ not chosen yet,
\begin{align*}
|\<s_m - s_k, \varphi \>| &\leq (2 \pi)^{N} \cdot \ds \sum_{k+1 \leq |\xi| \leq m} |a_{\xi}| \cdot |\hat{\varphi}(-\xi)| \\
&\leq (2 \pi)^{N} \cdot C_{1} \cdot C_{\varepsilon} \cdot  \ds \sum_{k+1 \leq |\xi| \leq m} \ds \sup_{n \in \N_{0}} \left(  \ds \frac{ \varepsilon^{n} \cdot (1 +|\xi|)^{n} } { m_{n} \cdot n!} \right)  \cdot \inf_{n \in \N_{0}}\left(\ds \frac{ m_n \cdot n!}{\delta^{n} (1+ |\xi|)^{n}} \right) 
\end{align*}
On the other hand, applying \eqref{Arrival} and \eqref{The Seventh Seal},
\begin{align*}
\ds \sup_{n \in \N_{0}} \left(  \ds \frac{ \varepsilon^{n} \cdot (1 +|\xi|)^{n} } { m_{n} \cdot n!} \right)  \cdot \inf_{n \in \N_{0}}\left(\ds \frac{ m_n \cdot n!}{\delta^{n} (1+ |\xi|)^{n}} \right) &= \left( \ds \frac{\varepsilon^{n_{0}} \cdot (1 +|\xi|)^{n_{0}} } {m_{n_{0}} \cdot n_{0}!} \right) \cdot \inf_{n \in \N_{0}}\left(\ds \frac{ m_n \cdot n!}{\delta^{n} (1+ |\xi|)^{n}} \right) \\
&\leq \left( \ds \frac{\varepsilon^{n_{0}} \cdot (1 +|\xi|)^{n_{0}} } {m_{n_{0}} \cdot n_{0}!} \right) \cdot \left(\ds \frac{ m_{(n_{0}+2N)} \cdot (n_{0}+2N)!}{\delta^{n_{0} + 2N} \cdot (1+ |\xi|)^{n_{0} + 2N}} \right) \\
&\leq \ds \frac{1} { \delta^{2N}}  \cdot \left(\ds \frac{\varepsilon \cdot B_{\left\{2N \right\}} \cdot C_{\left\{2N \right\}}}{\delta} \right)^{n_{0}}  \cdot \ds \frac{1}{(1+ |\xi|)^{2N}}
\end{align*}

Note that $n_{0}$ does depend on $\xi$ and $\varepsilon$. Nevertheless, if we choose $\varepsilon = \ds \frac{\delta}{C_{\left\{2N \right\}} \cdot B_{\left\{2N \right\}}}$ the dependence on $n_{0}$ disappears. Therefore if $C_{2} = \left(\ds \frac{(2\pi)^N \cdot C_{\varepsilon} \cdot C_{1}} {\delta^{2N}} \right)$, 
$$|\<s_m - s_k, \varphi \>| \leq C_{2} \cdot \ds \sum_{k+1 \leq |\xi| \leq m} \ds \frac{1}{(1+ |\xi|)^{2N}}.$$
Since the series on the right-hand side converges, $ \ds \lim_{k \to \infty} |\<s_m - s_k, \varphi\>| = 0$. So $\<s_{j}, \varphi \>$ is a Cauchy sequence for every $\varphi \in \E_{\M}(\T^N)$.  By the Lemma \ref{Magnolia} $u :=\ds \sum_{\xi \in \Z^{N}} a_{\xi} \cdot e^{i x \xi}$ is an element of $\D'_{\M}(\T^N)$.  It is easy to check that $\hat{u}(\xi) = a_{\xi}$. 
\end{proof}

\begin{Teo} \label{Shutter Island} 
Let $u \in \D_{\M}'(\T^N)$. Then 

$$u = \ds \sum_{\xi \in \Z^N} \hat{u}(\xi) \cdot e^{i x \xi},$$

\noindent with convergence in  $\D_{\M}'(\T^N)$. 
\end{Teo}

\begin{proof}
By Theorem $\ref{Memento}$, $\tilde{u} := \ds \sum_{\xi \in \Z^N} \hat{u}(\xi) \cdot e^{i x \xi}$ belongs to $\D_{\M}'(\T^N)$. Thus we just need to show that $u = \tilde{u}$. Given $\varphi \in \E_{\M}(\T^N)$, it follows from Theorem $\ref{Kramer vs Kramer}$ that

$$\varphi(x) = \ds \sum_{\xi \in \Z^N} \hat{\varphi}(\xi) \cdot e^{i x \xi}. $$

We define again $S_{k}\varphi(x) = \ds \sum_{|\xi| \leq k} \hat{\varphi}(\xi) \cdot e^{ix \xi}$ and since $S_{k}\varphi \to \varphi$ in $\E_{\M}(\T^N)$, it is not difficult to check that $\<u, \varphi \>  = \<\tilde{u}, \varphi \>$ and therefore $u = \tilde{u}$. 
\end{proof}

\begin{Teo} \label{The Aviator}
Let $\left\{b_{\xi} \right\}_{\xi \in \Z^N}$ be a sequence of complex numbers and suppose the existence of $C, \delta >0$ satisfying
$$|b_{\xi}| \leq C \cdot \inf_{n \in \N_{0}} \left(\ds \frac{m_{n} \cdot n!}{\delta^{n} \cdot (1+|\xi|)^n  } \right), \ \ \ \forall \xi \in \Z^N. $$
Then there exists $\psi \in \E_{\M}(\T^N)$ such that
$$\psi(x) = \ds \sum_{\xi \in \Z^N} b_{\xi} \cdot e^{i x \xi},$$
\noindent with convergence in $\E_{\M}(\T^N)$. In addition,  $\widehat{\psi}(\xi) = b_{\xi}$. 
\end{Teo}

\begin{proof} Let $\psi: \T^N \to \C$ defined as  $\psi (x) = \ds \sum_{\xi \in \Z^N} b_{\xi} \cdot e^{i x \xi}.$ By hypothesis, $\psi$ is smooth  and the series converges in $\E(\T^N)$. For any $\alpha \in \N_{0}^{N}$,  
\begin{align*}
\left|D^{\alpha} \psi(x)\right| &\leq \ds \sum_{\xi \in \Z^{N}} (1+\left|\xi \right|)^{\left|\alpha \right|} \cdot \left|b_{\xi} \right| \\
& \leq \ds \sum_{\xi \in \Z^{N}} \left[C \cdot \ds \frac{m_{|\alpha| + 2N} \cdot (|\alpha| +2N)!}{\delta^{|\alpha| + 2N}} \right] \cdot \ds \frac{1}{(1+\left|\xi \right|)^{2N}} \\
& \leq \left(C \cdot \left(\ds \frac{1}{\delta} \right)^{2N} \cdot \ds \sum_{\xi \in \Z^{N}} \ds \frac{1}{(1+\left|\xi \right|)^{2N}} \right) \cdot \left(\ds \frac{C_{\left\{2N \right\}} \cdot B_{\left\{2N \right\}}}{\delta}  \right)^{|\alpha|} \cdot m_{|\alpha|} \cdot |\alpha|!, 
\end{align*}

\noindent which shows us that $\psi \in \E_{\M}(\T^N)$. By denoting $S_{j}\Psi(x) = \ds \sum_{|\xi| \leq j} b_{\xi} \cdot e^{i x \xi}$ for each $j \in \N$, 
$$\left|D^{\alpha} (\Psi - S_{j} \Psi)(x) \right| \leq \left(C \cdot \left(\ds \frac{1}{\delta} \right)^{2N} \cdot \ds \sum_{|\xi| \geq j+1} \ds \frac{1}{(1+\left|\xi \right|)^{2N}} \right) \cdot \left(\ds \frac{C_{\left\{2N \right\}} \cdot B_{\left\{2N \right\}}}{\delta}  \right)^{|\alpha|} \cdot m_{|\alpha|} \cdot |\alpha|!.$$
Therefore $\ds \lim_{j \to \infty} \left\|\Psi - S_j \Psi \right\|_{\M, h_{1}} = 0$, which implies that $S_{j} \Psi \to \Psi$ in $\E_{\M}(\T^N)$.  
\end{proof}

\begin{Cor} \label{Forrest Gump}
Let $u \in \D_{\M}'(\T^N)$; if one can find $C, \delta >0$ such that
$$|\hat{u}(\xi)|  \leq C \cdot \inf_{n \in \N_{0}} \left(\ds \frac {m_{n} \cdot n!}{\delta^{n} \cdot (1+|\xi|)^{n}} \right), \ \ \ \ \forall \xi \in \Z^N,$$
then $u \in \E_{\M}(\T^N)$.
\end{Cor}

\begin{Obs}
The condition \eqref{Prisoners} was not used in the proofs of this section. We instead applied repeatedly \eqref{Arrival}, which is weaker. So those results remain valid for more general classes. 
\end{Obs}

\section{Partial Fourier Series}\label{The Phantom Menace} 

Let $R, S \in \N$ such that $R+S = N$; we denote $(t,x) \in \T^N$, where $t \in \T^{R}$ and $x \in \T^{S}$. Given $\varphi \in \E_{\M}(\T^N)$, consider for each $t \in \T^{R}$ 
\begin{eqnarray*}
	\varphi_{t}: \T^{S} &\to & \C  \\
	x & \mapsto & \varphi(t,x).
\end{eqnarray*}
Then $\varphi_{t} \in \E_{\M}(\T^{S})$ for any $t$ and it follows from Theorem \ref{Kramer vs Kramer} that

$$\varphi(t,x) = \varphi_{t}(x) = \ds \sum_{\eta \in \Z^S} \widehat{\varphi}(t, \eta) \cdot e^{i x \eta}, $$

\noindent with convergence in $\E_{\M}(\T^{S})$, and
$$\widehat{\varphi}(t, \eta) = \ds \frac{1}{(2 \pi)^{S}} \ds \int_{\T^{S}} \varphi_{t}(x)  \cdot e^{-ix \eta} dx = \ds \frac{1}{(2 \pi)^{S}} \ds \int_{\T^{S}} \varphi (t, x) \cdot e^{-ix \eta} dx.$$

It is easy to check that $\widehat{\varphi}(t, \eta) \in \E_{\M}(\T^{R})$; let us prove a sharper estimate. 

\begin{Teo} \label{Children of Men}
Given $\varphi \in \E_{\M}(\T^N)$, there exists  $C, h, \delta > 0$ such that	
$$|\partial_{t}^{\alpha}\widehat{\varphi}(t, \eta)| \leq C \cdot h^{|\alpha|} \cdot m_{|\alpha|} \cdot |\alpha|! \cdot \left(\ds \inf_{p \in \N_{0}} \ds \frac{m_{p} \cdot p!}{\delta^{p} \cdot (1+|\eta|)^{p}} \right), \ \ \ \forall t \in \T^{R}, \ \ \forall \alpha \in \N_{0}^{R}, \ \ \forall \eta \in \Z^{S}.$$
\end{Teo}

\begin{proof} 
Note first that
$$\partial_{t}^{\alpha}\widehat{\varphi}(t, \eta) = \ds \frac{1}{(2 \pi)^{S}} \ds \int_{\T^{S}} \partial_{t}^{\alpha}\varphi (t, x) \cdot e^{-ix \eta} dx = \widehat{\partial_{t}^{\alpha}\varphi}(t, \eta). $$	
\noindent Thus $\partial_{t}^{\alpha} \widehat{D_{x}^{\beta} \varphi}(x, \eta) = \eta^{\beta} \cdot \partial_{t}^{\alpha}\widehat{\varphi}(t, \eta)$ and 	
\begin{align*}
|\eta^{\beta} \cdot \partial_{t}^{\alpha}\widehat{\varphi}(t, \eta)| &\leq \ds \sup_{(t,x) \in \T^N} \left|\partial_{t}^{\alpha} D_{x}^{\beta}\varphi (t, x) \right| \\
&\leq \left\| \varphi \right\|_{\M, h_{1}} \cdot h_{1}^{|\alpha| + |\beta|} \cdot m_{|\alpha| + |\beta|} \cdot (|\alpha| + |\beta|)! \\
& \leq \left\| \varphi \right\|_{\M, h_{1}} \cdot \left[(2h_{1}H)^{|\alpha|} \cdot m_{|\alpha|} \cdot |\alpha|! \right] \cdot \left[(2h_{1}H)^{|\beta|} \cdot m_{|\beta|} \cdot |\beta|! \right], 
\end{align*}
\noindent applying \eqref{Prisoners}.
		
Given $\eta \neq 0$ and $p \in \N_{0}$ we choose $\beta \in \N_{0}^{S}$ such that $|\beta| = p$ and $\ds \frac{|\eta|^{p}}{|\eta^{\beta}|} \leq N^{p/2}$ . Thus
	
$$(1 + |\eta|)^{p} \cdot |\partial_{t}^{\alpha}\widehat{\varphi}(t, \eta)| \leq \left\| \varphi \right\|_{\M, h_{1}} \cdot \left[(2h_{1}H)^{|\alpha|} \cdot m_{|\alpha|} \cdot |\alpha|! \right] \cdot \left[(4h_{1}H \sqrt{N})^{p} \cdot m_{p} \cdot p! \right].$$
	
\noindent Therefore, by taking $h = 2h_{1}H$,  $\delta = \ds \frac{1}{4h_{1}H \sqrt{N}}$ and $C =  \left\| \varphi \right\|_{\M, h_{1}} \cdot \left(\ds \inf_{p \in \N_{0}} \ds \frac{m_{p} \cdot p!}{\delta^{p}} \right)^{-1}$, we also include the case $\eta = 0$. 
\end{proof}

\begin{Teo} \label{Brooklyn}
Suppose that for each $\eta \in \Z^{S}$ there exists a function $\varphi_{\eta}$ in $\E_{\M}(\T^R)$ and positive constants $C, h, \delta$ such that 	
$$|\partial_{t}^{\alpha} \varphi_{\eta}(t)| \leq C \cdot h^{|\alpha|} \cdot m_{|\alpha|} \cdot |\alpha|! \cdot \ds \inf_{p \in \N_{0}} \left( \ds \frac{m_{p} \cdot p!}{\delta^{p} \cdot (1+|\eta|)^{p}} \right), \ \ \ \forall t \in \T^{R}, \ \ \forall \alpha \in \N_{0}^{R}, \ \ \forall \eta \in \Z^{S}.$$
	
\noindent Then $\varphi: \T^N \to \C$ defined as $\varphi(t,x) = \ds \sum_{\eta \in \Z^{S}} \varphi_{\eta}(t) \cdot e^{ix \eta}$ is an element of $ \E_{\M}(\T^N)$. 
\end{Teo}

\begin{proof}
It is easy to check that $\varphi$ is well defined. Let us show that $\varphi \in \E_{\M}(\T^N)$: 
\begin{align*}
\left|\partial_{t}^{\alpha_{1}} \partial_{x}^{\alpha_{2}} \varphi(t,x)\right| &\leq \ds \sum_{\eta \in \Z^{S}} \left|\partial_{t}^{\alpha_{1}} \varphi_{\eta}(t) \right| \cdot (1+|\eta|)^{|\alpha_{2}|} \\  
&\leq  \ds \sum_{\eta \in \Z^{S}}  C \cdot h^{|\alpha_{1}|} \cdot m_{|\alpha_{1}|} \cdot |\alpha_{1}|! \left(\ds \frac{m_{|\alpha_{2}| + 2N} \cdot (|\alpha_{2}| + 2N)!}{\delta^{|\alpha_{2}| + 2N} \cdot (1+|\eta|)^{|\alpha_{2}| + 2N}} \right) \cdot (1+|\eta|)^{|\alpha_{2}|} \\
&\leq \ds \frac{C}{\delta^{2N}}\cdot  h^{|\alpha_{1}|}  \left(\ds \frac{B_{\left\{2N \right\}} \cdot C_{\left\{2N \right\}}}{\delta} \right)^{|\alpha_{2}|} m_{|\alpha_{1}+\alpha_{2}|} \cdot (|\alpha_{1}+\alpha_{2}|)!  \ds \sum_{\eta \in \Z^{S}} \ds \frac{1}{(1+|\eta|)^{2N}}.
\end{align*}
Thus if we take $C_{1} = \ds \frac{C}{\delta^{2N}} \cdot \ds \sum_{\eta \in \Z^{S}} \ds \frac{1}{(1+|\eta|)^{2N}}$ and $h_{1} = \max \left\{h, \ds \frac{B_{\left\{2N \right\}} \cdot C_{\left\{2N \right\}}}{\delta}  \right\}$,  
	
$$\left|\partial_{t}^{\alpha_{1}} \partial_{x}^{\alpha_{2}} \varphi(t,x)\right| \leq C_{1} \cdot h_{1}^{|\alpha_{1}+\alpha_{2}|} \cdot m_{|\alpha_{1}+\alpha_{2}|} \cdot (|\alpha_{1}+\alpha_{2}|)!,$$
which ends our proof.	
\end{proof}

The next step is to develop Partial Fourier Series for ultradistributions. Given $u \in \D_{\M}'(\T^N)$, we define for each  $\psi \in \E_{\M}(\T^R)$
\begin{eqnarray*}
	u_{\psi}: \T^{S} &\to & \C  \\
	\varphi & \mapsto & \<u, \psi \otimes \varphi \>, 
\end{eqnarray*}
where $\psi \otimes \varphi (t,x) = \psi(t) \cdot \varphi(x)$.

\begin{Lem} \label{Amour}
For every $u \in \D_{\M}(\T^N), \psi \in \E_{\M}(\T^R)$, we have that $u_{\psi} \in \D_{\M}'(\T^S)$. 
\end{Lem}

\begin{proof}
Let $\left\{\varphi_{n} \right\}_{n \in \N}$ be a sequence in $\E_{\M}(\T^S)$ converging to $0$. Since $u$ is continuous, it is sufficient to show that $\psi \otimes \varphi_{n} \to 0$ in $\E_{\M}(\T^N)$. By our assumption, one can find $h_{S} > 0$ such that	
$$\varphi_{n} \in \E_{\M, h_{S}}(\T^S), \ \forall n \in \N; \ \  \left\|\varphi_{n} \right\|_{\M, h_{S}} \to 0. $$
	
Moreover, there exists $h_{R} > 0$ such that $\psi \in \E_{\M, h_{R}}(\T^R)$. Thus
\begin{align*}
\ds \sup_{(t, x) \in \T^N} \left|\partial_{t}^\alpha \partial_{x}^\beta (\psi(t) \cdot \varphi(x)) \right| &\leq \ds \sup_{t \in \T^{R}} \left|\partial_{t}^\alpha \psi(t) \right| \cdot \ds \sup_{x \in \T^{S}} \left|\partial_{x}^\beta \varphi(x) \right| \\
& \leq \left\| \psi \right\|_{\M, h_{R}} \cdot h_{R}^{|\alpha|} \cdot m_{|\alpha|} \cdot |\alpha|! \cdot \left\|\varphi_{n} \right\|_{\M, h_{S}} \cdot h_{S}^{|\beta|} \cdot m_{|\beta|} \cdot |\beta|!. 
\end{align*}

By taking $h = \max \left\{h_{R}, h_{S} \right\}$, it follows from the inequality above that 
	$$\psi \otimes \varphi_{n} \in \E_{\M, h}(\T^N), \ \ \forall n \in \N.$$
	$$ \left\| \psi \otimes \varphi_{n} \right\|_{\M, h} \leq \left\|\psi \right\|_{\M, h_{R}} \cdot \left\|\varphi_{n} \right\|_{\M, h_{S}} \to 0, $$ 
	
\noindent which proves the result. 
\end{proof}

By Theorem \ref{Shutter Island}, we can write  

$$u_{\psi} = \ds \sum_{\eta \in \Z^{S}} \widehat{u_{\psi}}(\eta) \cdot e^{i x \eta}, $$

\noindent where $\widehat{u_{\psi}}(\eta) = \ds \frac{1}{(2 \pi)^{S}} \cdot \< u_{\psi}, e^{-ix \eta} \> = \ds \frac{1}{(2 \pi)^{S}} \cdot \< u, \psi(t) \cdot e^{-ix \eta} \>.$ Now for each $\eta \in \Z^{S}$ we define the following functional:

\begin{equation} \label{12 Years a Slave}
\<\widehat{u}(t, \eta), \psi \> = \widehat{u_{\psi}}(\eta) = \ds \frac{1}{(2 \pi)^{S}} \cdot \< u, \psi(t) \cdot e^{-ix \eta} \>. 
\end{equation}

\begin{Obs}
The notation used  above is analogous to the one applied for ultradifferentiable functions,  but $\widehat{u}(t, \eta)$ is not necessarily a function.  
\end{Obs}

\begin{Lem} \label{Carol}
For every $\eta \in \Z^{S}$, $\widehat{u}(t, \eta)$ defined in \eqref{12 Years a Slave} belongs to $\D_{\M}'(\T^{R})$.   
\end{Lem}

\begin{proof}
Let  $\left\{\psi_{n} \right\}_{n \in \N}$ be a sequence converging to  $0$ in $\E_{\M}(\T^R)$. It is sufficient to verify that 
$$\psi_{n} \cdot e^{-i x \eta} \to 0 \ \text{in} \ \E_{\M}(\T^N), \ \text{for any} \ \eta. $$ 
Since $e^{-ix \eta} \in C^{\omega}(\T^{S})$, it also belongs to $\E_{\M}(\T^{S})$. We proceed then  as in Lemma \ref{Amour}.  
\end{proof}

\begin{Teo} \label{The Royal Tenembauns}
Let $u \in \D_{\M}(\T^N)$. Then 
	
$$u = \ds \sum_{\eta \in \Z^S} \widehat{u}(t, \eta) \cdot e^{i x \eta},$$
	
\noindent with convergence in $\D_{\M}'(\T^N)$. Moreover, for every $\varepsilon, h > 0$ there exists $C_{\varepsilon, h} > 0$ such that
\begin{equation} \label{Zero Dark Thirty}
\left|\< \widehat{u}(t, \eta), \psi \> \right| \leq C_{\varepsilon, h} \cdot \left\|\psi \right\|_{\M, h} \cdot \ds \sup_{p \in \N_{0}} \left( \ds \frac{\varepsilon^{p} \cdot (1+|\eta|)^{p}}{m_{p} \cdot p!} \right), \ \ \forall \eta \in \Z^{S},\ \ \forall \psi \in \E_{\M, h}(\T^{R}). 
\end{equation}	
\end{Teo}

\begin{proof}
Let $\lambda (t,x) \in \E_{\M}(\T^N)$; since $\lambda (t,x) = \ds \sum_{\eta \in \Z^{S}} \widehat{\lambda}(t, \eta) \cdot e^{i \eta x}$, with convergence in $\E_{\M}(\T^N)$, we have
\begin{align*}
\<u, \lambda \> &= \ds \lim_{k \to + \infty} \ds \sum_{|\eta| \leq k} \<u, \widehat{\lambda}(t, \eta) \cdot e^{i  x \eta} \> \\
&= \lim_{k \to + \infty} \ds \sum_{|\eta| \leq k}  \<\widehat{u}(t, \eta), \ds \int_{\T^{S}} \lambda (t, x) \cdot e^{ix \eta} dx \> \\
&= \lim_{k \to + \infty} \ds \sum_{|\eta| \leq k}  \<\widehat{u}(t, \eta) \cdot e^{ix \eta}, \lambda \>. 
\end{align*}

Let us proceed to the proof of \eqref{Zero Dark Thirty}. Given $\varepsilon, h > 0$;  it follows from Theorem \ref{The Hurt Locker}, the existence of  $C_{\varepsilon}$ such that
\begin{align*}
\left| \< \widehat{u}(t, \eta), \psi \> \right| &\leq \ds \frac{C_{\varepsilon}}{(2 \pi)^{S}} \cdot  \underset{(\alpha, \beta) \in \N_{0}^{N}} {\ds \sup_{(t,x) \in \T^N}} \left( \ds \frac{\left|\partial_{t}^{\alpha} \partial_{x}^{\beta} \left(\psi(t) \cdot e^{-i x \eta} \right)\right| \cdot \varepsilon^{|\alpha + \beta|}}{m_{|\alpha + \beta|} \cdot (|\alpha + \beta|) !} \right) \\
& \leq \ds \frac{C_{\varepsilon}}{(2 \pi)^{S}} \cdot \left\|\psi \right\|_{\M, h} \cdot \underset{(\alpha, \beta) \in \N_{0}^{N}} {\ds \sup_{(t,x) \in \T^N}} \left( \ds \frac{h^{|\alpha|} \cdot (1+|\eta|)^{|\beta|} \cdot \varepsilon^{|\alpha|} \cdot \varepsilon^{|\beta|}} { m_{|\beta|} \cdot |\beta|!} \right).
\end{align*}

When $\varepsilon \leq \ds \frac{1}{h}$,
\begin{align*}
\left| \< \widehat{u}(t, \eta), \psi \> \right| &\leq \ds \frac{C_{\varepsilon}}{(2 \pi)^{S}} \cdot \left\|\psi \right\|_{\M, h} \cdot  \ds \sup_{\beta \in \N_{0}^{S}} \left( \ds \frac{\varepsilon^{|\beta|} \cdot (1+|\eta|)^{|\beta|} } { m_{|\beta|} \cdot |\beta|!} \right) \\
&\leq \ds \frac{C_{\varepsilon}}{(2 \pi)^{S}} \cdot \left\|\psi \right\|_{\M, h}  \cdot  \ds \sup_{p \in \N_{0}}\left( \ds \frac{\varepsilon^{p} \cdot (1+|\eta|)^{p}}{m_{p} \cdot p!} \right).
\end{align*}
Otherwise,
$$\left| \< \widehat{u}(t, \eta), \psi \> \right| \leq \ds \frac{C_{1/h}}{(2 \pi)^{S}} \cdot \left\|\psi \right\|_{\M, h}  \cdot  \ds \sup_{p \in \N_{0}} \left( \ds \frac{ (1+|\eta|)^{p}}{h^{p} \cdot m_{p} \cdot p!} \right) \leq \ds \frac{C_{1/h}}{(2 \pi)^{S}} \cdot \left\|\psi \right\|_{\M, h}   \cdot \ds \sup_{p \in \N_{0}} \left( \ds \frac{\varepsilon^{p} \cdot (1+|\eta|)^{p}}{ m_{p} \cdot p!} \right).$$
	
\noindent Therefore, we choose $C_{h, \varepsilon} = \ds \frac{C_{\varepsilon}}{(2 \pi)^{S}}$ for the first case and $C_{h, \varepsilon} = \ds \frac{C_{1/h}}{(2 \pi)^{S}}$ for the second one. 
\end{proof}

\begin{Teo} \label{A Clockwork Orange}
Consider $\left\{u_{\eta} \right\}_{\eta \in \Z^S}$ a sequence in $\D_{\M}'(\T^R)$, satisfying the following condition:  for every $\varepsilon, h > 0$, we can find $C_{\varepsilon, h} > 0$ such that 
	
$$\left|\< u_{\eta}, \psi \> \right| \leq C_{\varepsilon, h} \cdot \left\|\psi \right\|_{\M, h} \cdot \left(\ds \sup_{p \in \N_{0}} \ds \frac{\varepsilon^{p} \cdot (1+|\eta|)^{p}}{m_{p} \cdot p!} \right), \ \ \forall \eta \in \Z^{S},\ \ \forall \psi \in \E_{\M, h}(\T^{R}).$$ 
	
\noindent Then,  $u:= \ds \sum_{\eta \in \Z^S} u_{\eta} \cdot e^{i x \eta}$ is an element of  $\D_{\M}'(\T^N)$. 
\end{Teo}

\begin{proof}
Given $\lambda (t,x) \in \E_{\M}(\T^N)$, there exists $r > 0$ such that $\lambda \in \E_{\M, r}(\T^N)$. We define
$$s_{j} = \ds \sum_{|\eta| \leq j} u_{\eta} \cdot e^{i x \eta}. $$
Note that $\<s_{j}, \lambda \> = \ds \sum_{|\eta| \leq j} \<u_{\eta} \cdot e^{i x \eta}, \lambda \> = \ds \sum_{|\eta| \leq j} \< u_{\eta}, \ds \int_{\T^{S}} \lambda(t,x) \cdot e^{i x \eta} dx \>$. Hence, if $k \in \N$, it follows from Theorem \ref{Children of Men} that
\begin{align*}
\left|\<s_{k+j} - s_{j}, \lambda \> \right| &\leq (2\pi)^{S} \cdot \ds \sum_{j < |\eta| \leq j+k} \left|\<u_{\eta}, \widehat{\lambda}(t,-\eta) \> \right| \\
&\leq (2\pi)^{S} \cdot \ds \sum_{j < |\eta| \leq j+k} C_{\kappa, 2rH} \cdot \left\|\widehat{\lambda}( \cdot, -\eta) \right\|_{\M, 2rH} \cdot \left(\ds \sup_{p \in \N_{0}} \ds \frac{\kappa^{p} \cdot (1+|\eta|)^{p}}{m_{p} \cdot p!} \right) \\
&\leq  \ds \sum_{j < |\eta| \leq j+k} C_{\kappa, \delta, 2rH}'  \cdot \left(\ds \inf_{p \in \N_{0}} \ds \frac{m_{p} \cdot p!}{\delta^{p} \cdot (1+|\eta|)^{p}} \right)  \cdot \left(\ds \sup_{p \in \N_{0}} \ds \frac{\kappa^{p} \cdot (1+|\eta|)^{p}}{m_{p} \cdot p!} \right),
\end{align*}
with $\kappa > 0$ which has yet to be defined. 
	
By taking $\kappa = \ds \frac{1}{H \delta}$, with $H$ as in \eqref{Prisoners}, it follows from Lemma \ref{Jaws} that 
\begin{align*}
\left|\<s_{k+j} - s_{j}, \lambda \> \right| & \leq  \ds \sum_{j < |\eta| \leq j+k} C_{\delta, \kappa, 2rH}'  \cdot \left(\ds \inf_{p \in \N_{0}} \ds \frac{m_{p} \cdot p!}{\delta^{p} \cdot (1+|\eta|)^{p}} \right)^{1/2} \\
& \leq \ds \sum_{j < |\eta| \leq j+k} C_{\delta, \kappa,  2rH}'  \cdot \left(\ds \frac{(m_{4N} \cdot 4N!)^{1/2}}{\delta^{2N} \cdot (1+|\eta|)^{2N}} \right) \\
&\leq C_{\delta, \kappa,  2rH}''  \cdot \ds \sum_{j < |\eta| \leq j+k} \ds \frac{1}{(1+|\eta|)^{2N}}.
\end{align*}

Hence $\ds \lim_{j \to \infty} \left|\<s_{k+j} - s_{j}, \lambda \> \right| = 0$, which shows us that $\< s_{j}, \lambda \>$ is a Cauchy sequence and therefore (Lemma \ref{Magnolia}) $u \in \D_{\M}(\T^N)$. 
	
\end{proof}

\section{Applications} \label{Revenge of the Sith} 

\subsection{Systems of Constant Coefficient Operators}

Just like in the Analytic (\cite{greenfield1972hypoelliptic}) and Gevrey (\cite{himonas2004gevrey}) cases, we will prove that global hypoellipticity for systems of constant coefficient operators are directly related to the weight function associated. Let $P_{1}(D), P_{2}(D), \cdots P_{k}(D)$ be constant coefficients linear partial differential operators acting on $\T^N$ and consider the system
\begin{equation} \label{Gandhi} 
P_{j} u = f_{j}, \ \ \  u \in \D_{\M}'(\T^N), \ \ f_{j} \in \D_{\M}'(\T^N), \ \ j = 1, 2, \ldots, k. 
\end{equation}

\noindent We define its \textbf{symbol} as $P(\xi) := \left(P_{1}(\xi), P_{2}(\xi), \ldots, P_{k}(\xi) \right)$ for each $\xi \in \Z^N$ and $|P(\xi)| := \ds \max_{1 \leq j \leq k} |P_{j}(\xi)|. $

\begin{Def}
We say that \eqref{Gandhi} is  \textbf{globally $\M$-hypoelliptic} if
$$u \in \D_{\M}'(\T^N), \ f_{j} \in \E_{\M}(\T^N), \ P_{j}(D) u = f_{j}, \ j= 1, 2, \ldots, k \ \Rightarrow u \in \E_{\M}(\T^N).$$
\end{Def}

\begin{Teo} \label{Spider Man II}
The system \eqref{Gandhi} is globally $\M$-hypoelliptic if and only if for every $\varepsilon > 0$, there exists $R_{\varepsilon} > 0$ such that 

\begin{equation} \label{Harry Potter and the Philosopher's Stone}
|P(\xi)| \geq  \ds \inf_{n \in \N_{0}} \left(\ds \frac{m_{n} \cdot n!} {\varepsilon^{n} \cdot (1+|\xi|)^{n}}\right), \ \ \ \forall \xi \in \Z^N; \ |\xi| \geq R_{\varepsilon},
\end{equation} 
\end{Teo}

\begin{proof}
We start by proving the sufficiency: let $u \in \D_{\M}'(\T^N)$ such that  $P_{j}(D)u = f_{j} \in \E_{\M}(\T^N)$, for $j = 1, 2, \ldots, k$. It follows from Theorem \ref{Kramer vs Kramer} the existence of $C, \delta > 0$ satisfying 
$$|\widehat{f}_{j}(\xi)| \leq C \cdot \ds \inf_{n \in \N_{0}} \left( \ds \frac{m_{n} \cdot n!}{\delta^{n} \cdot (1+|\xi|)^{n}} \right), \ \ \forall \xi \in \Z^{N}, \ \ j = 1, 2, \ldots, k.$$ 
We write $u = \ds \sum_{\xi \in \Z^N} \widehat{u}(\xi) \cdot e^{i x \xi}$ and take $\varepsilon = \ds \frac{\delta}{H}$. Given $\xi \in \Z$ with $ |\xi| \geq R_{\delta/H}$, 
\begin{align*}
|\widehat{u}(\xi)| &\leq C \left(\ds \inf_{n \in \N_{0}} \ds \frac{m_{n} \cdot n!}{\delta^{n} \cdot (1+|\xi|)^{n}} \right) \cdot \left[\ds \sup_{n \in \N_{0}} \ds \frac{(\delta)/H)^{n} \cdot (1+|\xi|)^{n}}{m_{n} \cdot n!}  \right] \\
&\leq C \cdot \left(\ds \inf_{n \in \N_{0}} \ds \frac{m_{n} \cdot n!}{\varepsilon^{n} \cdot (1+|\xi|)^{n}} \right),
\end{align*}
by Lemma \ref{Jaws}. Hence, by possibly increasing C, we obtain 
$$|\widehat{u}(\xi)| \leq C \cdot \ds \inf_{n \in \N_{0}} \left( \ds \frac{m_{n} \cdot n!}{\varepsilon^{n} \cdot (1+|\xi|)^{n}} \right), \ \forall \xi \in \Z^N,$$

\noindent which allows us to infer (Corollary \ref{Forrest Gump}) that $u \in \E_{\M}(\T^N)$. 

Let us proceed to the necessity. If \eqref{Harry Potter and the Philosopher's Stone} does not hold, there exist $\varepsilon > 0$ and a sequence $\left\{\xi_{m} \right\}_{m \in \N}$ in $\Z^N$ satisfying the following properties:
$$|P(\xi_{m})| < \ds \inf_{n \in \N_{0}} \left(\ds \frac{m_{n} \cdot n!}  {\varepsilon^{n} \cdot (1+|\xi|)^{n}} \right), \ |\xi_{m}| \geq m, \ \ \  \forall m \in \N.$$
Then $u = \ds \sum_{m \in \N} e^{i x \cdot \xi_{m} }$ is a ultradistribution (Theorem \ref{Memento}) but not a smooth function.

On the other hand, $f_{j} := P_{j}(D)u = \sum_{m \in \N} P_{j}(\xi_{m}) \cdot e^{i x \cdot \xi_{m}}$ and 
$$|P_{j}(\xi_{m})| \leq |P(\xi_{m})| < \ds \inf_{n \in \N_{0}} \left(\ds \frac{m_{n} \cdot n!}  {\varepsilon^{n} \cdot (1+|\xi|)^{n}} \right), \ \forall m \in \N, \ \  j = 1, 2, \ldots, k. $$

\noindent It follows from Theorem \ref{The Aviator} that  $f_{j} \in \E_{\M}(\T^N)$  for $j \in \left\{1, 2, \ldots, k \right\}$, as we intended to prove.  

\end{proof} 

\begin{Cor} \label{Roma}
Let $P(D)$ be a system of constant coefficients linear partial differential operators that is globally $C^{\infty}$-hypoelliptic. Then it is globally $\M$-hypoelliptic.  
\end{Cor}

\begin{proof}
By Greenfield-Wallach's condition (\cite{greenfield1972global}), global smooth hypoellipticity of $P$ implies the existence of $L, k, R > 0$ such that 
$$|P(\xi)| \geq \ds \frac{L}{(1+ |\xi|)^{k}}, \ \ |\xi| \geq R.$$
Note that we may consider $k \in \N$. Given $\varepsilon > 0$, we take $R_{\varepsilon} = \ds \max \left\{R, \ \ds \frac{m_{k+1} \cdot (k+1)!}{L \cdot \varepsilon^{k+1}} \right\}$ and obtain that 
$$|\xi| \geq R_{\varepsilon} \ \Rightarrow |P(\xi)| \geq \ds \frac{L}{(1+|\xi|)^{k}} = \ds \frac{L \cdot (1+|\xi|)}{(1+|\xi|)^{k+1}} \geq \ds \frac{m_{k+1} \cdot (k+1)!}{\varepsilon^{k+1} \cdot (1+|\xi|)^{k+1}} \geq \ds \inf_{n \in \N_0} \left(\ds \frac{m_n \cdot n!}{\varepsilon^n \cdot (1+|\xi|)^n} \right), $$
which ends our proof.
\end{proof}

\begin{Cor} \label{Harry Potter and the Chamber of Secrets}
Let $\M$ and $\mathscr{L}$ be weight sequences such that $\E_{\M}(\T^N) \subset \E_{\mathscr{L}}(\T^N)$. If the system \eqref{Gandhi} is globally $\mathscr{L}$-hypoelliptic, it is also globally $\M$-hypoelliptic. 
\end{Cor}

\begin{proof}
By \eqref{The Godfather} and \eqref{The Godfather II}, there exists $C \geq 1$ such that
$$m_{k} \leq C^{k} \cdot \ell_{k}, \ \forall k \in \N_{0}. $$
Given $\varepsilon > 0$, we find $R_{\varepsilon}$ satisfying
$$|P(\xi)| \geq  \ds \inf_{n \in \N_{0}}\left( \ds \frac {C^{n} \cdot \ell_{n} \cdot n!}{\varepsilon^{n} \cdot (1+|\xi|)^{n}}  \right), \ \ \ \forall \xi \in \Z^N; \ |\xi| \geq R_{\varepsilon}.$$

On the other hand, for every $\xi \in \Z^N$ such that $|\xi| \geq R_{\varepsilon}$, we have
$$\inf_{n \in \N_{0}} \left( \ds \frac{C^{n} \cdot \ell_{n} \cdot n!}{\varepsilon^{n} \cdot (1+|\xi|)^{n}}  \right) \geq \inf_{n \in \N_{0}} \left(\ds \frac{m_{n} \cdot n!} {\varepsilon^{n} \cdot (1+|\xi|)^{n}}  \right) \ \Rightarrow  |P(\xi)| \geq \inf_{n \in \N_{0}} \left(\ds \frac{m_{n} \cdot n!} {\varepsilon^{n} \cdot (1+|\xi|)^{n}}  \right) , $$
which concludes the proof. 
\end{proof}

\subsection{Greenfield-Wallach Vector Fields}

Analogously to what was done in \cite{greenfield1972hypoelliptic}, \cite{greenfield1972global} and \cite{yoshino}, we are able to apply Theorem \ref{Spider Man II} in order to study global $\M$-hypoellipticity of the following system acting on $\T^{2}$:
\begin{equation} \label{The Shawshank Redemption}
P_{\alpha}(D_{1}, D_{2}) = D_{1} - \alpha D_{2}, \ \ \  \ \ \ \alpha \in \C.   
\end{equation}

\noindent and to extend the connections with Number Theory.  By Corollary \ref{Roma}, the interesting cases occur when $\alpha \in \R$.

\begin{Def}
We say $\alpha \in \R \setminus \Q$ is \textbf{Liouville $\M$-exponential} if one can find $\varepsilon > 0$ such that the inequality
$$|\xi - \alpha \eta| < \ds \inf_{n \in \N_{0}} \left( \ds \frac {m_{n} \cdot n!} {\varepsilon^{n} \cdot (1+|\eta|)^{n}} \right), \ \ \ (\xi, \eta) \in \Z \times \Z,$$

\noindent has infinite solutions. 
\end{Def}

\begin{Pro} \label{The Silence of the Lambs} 
Let $\alpha$ be a real number; $P_{\alpha}(D_{1}, D_{2})$ is $\M$-hypoelliptic if and only if $\alpha$ is irrational \textbf{non}-Liouville $\M$-exponential.  
\end{Pro}

\begin{proof}
Since the symbol of $P_\alpha$ is given by
$$P_{\alpha}(\xi, \eta) = \ds \frac{1}{i} \cdot \left(\xi - \alpha \eta \right), \ \ (\xi, \eta) \in \Z \times \Z,  $$ 
when $\alpha \in \Q$ it is possible to obtain a sequence $\left\{\xi_{m}, \eta_{m} \right\}_{m \in \N}$ such that $P_{\alpha}(\xi_{m}, \eta_{m}) = 0$, for each natural $m$. Thus  
$$|P_{\alpha}(\xi_{m}, \eta_{m})| < \ds \inf_{n \in \N_{0}}  \left(\ds \frac{m_{n} \cdot n!}{\varepsilon^{n} \cdot (1+|\xi_{m}| + |\eta_{m}|)^{n}}  \right), \ \ \forall m \in \N. $$

\noindent By Theorem \ref{Spider Man II}, $P_{\alpha}(D_{1}, D_{2})$ is not globally $\M$-hypoelliptic.

We proceed to the case where $\alpha$ is irrational; if $\alpha$ is non-Liouville $\M$-exponential, for every $\varepsilon > 0$, one can find $R_{\varepsilon} > 0$ such that

$$|\xi - \alpha \eta| \geq \ds \inf_{n \in \N_{0}} \left( \ds \frac{m_{n} \cdot n!} {\varepsilon^{n} \cdot (1+|\eta|)^{n}} \right), \ \ |\xi| + |\eta| \geq R_{\varepsilon}. $$ 
Hence 
\begin{equation} \label{Harry Potter and the Prisoner of Azkaban}
\forall \varepsilon > 0, \ \exists R_{\varepsilon}>0; \ \ \ |\xi - \alpha \eta| \geq  \ds \inf_{n \in \N_{0}} \left(\ds \frac{m_{n} \cdot n!} {\varepsilon^{n} \cdot (1+|\xi| + |\eta|)^{n}} \right), \ \ |\xi| + |\eta| \geq R_{\varepsilon},
\end{equation}
and $P_\alpha$ is $\M$-hypoelliptic (Theorem \ref{Spider Man II}). 

On the other hand, suppose $P_{\alpha}$ globally $\M$-hypoelliptic. It follows from Theorem  \ref{Spider Man II} that \eqref{Harry Potter and the Prisoner of Azkaban} holds.  For any $\delta > 0$, we take $\varepsilon = \ds \frac{\delta}{(|\alpha| + 1)}$. In the situation where $|\xi - \alpha \eta| > 1$,  

$$ |\xi - \alpha \eta| > 1 \geq  \ds \inf_{n \in \N_{0}} \left( \ds \frac{m_{n} \cdot n!}{\varepsilon^{n} \cdot (1+|\eta|)^{n}}  \right).$$

\noindent Otherwise $|\alpha| \cdot |\eta| - |\xi| \leq 1 \ \Rightarrow |\alpha| \cdot |\eta| \leq (1+|\xi|)$. So, when $|\xi| + |\eta| \geq R_{\varepsilon},$
\begin{align*}
|\xi - \alpha \eta| &\geq \ds \inf_{n \in \N_{0}} \left( \ds \frac{m_{n} \cdot n!} {\varepsilon^{n} \cdot (1+|\xi| + |\eta|)^{n}} \right) \\
& \geq \ds \inf_{n \in \N_{0}} \left( \ds \frac{m_{n} \cdot n!} {(\varepsilon \cdot (|\alpha| + 1))^{n} \cdot  (|\eta|)^{n}} \right)  \\
& \geq \ds \inf_{n \in \N_{0}} \left( \ds \frac{m_{n} \cdot n!}{\delta^{n} \cdot (1 +|\eta|)^{n}}  \right),
\end{align*}
since we may consider $\eta \neq 0$, which shows that $\alpha$ is non-Liouville $\M$-exponential. 
\end{proof}

We verified in Corollary \ref{Harry Potter and the Chamber of Secrets} that if $\M$, $\mathscr{L}$ are weight sequences, $\E_{\M}(\T^N) \subset \E_{\mathscr{L}}(\T^N)$ and $P_\alpha$ is globally $\mathscr{L}$-hypoelliptic, then it is also globally $\M$-hypoelliptic.  On the other hand, it is proved in \cite{yoshino} that given $r, s \geq 1$ with $r < s$, one can find $\alpha \in \R \setminus \Q$ such that  $P_{\alpha}$ is globally $\mathcal{G}^{r}$-hypoelliptic, but not globally $\mathcal{G}^{s}$-hypoelliptic. Our goal here is to extend this result to Denjoy-Carleman Classes in general.

\begin{Def}
Let $ \mathscr{L} = \left\{\ell_{n} \right\}_{n \in \N_{0}}, \M = \left\{m_{n} \right\}_{n \in \N_{0}}$ be arbitrary weight sequences.  We will denote $\M \prec  \mathscr{L} $ if $\ds\lim_{k \to +\infty} \left(\ds \frac{m_{k}}{\ell_{k}} \right)^{1/k} = 0$.  
\end{Def}

\begin{Obs} \label{Mission Impossible: Ghost Protocol}
By notation set in \eqref{The Godfather}, when $\M \prec \mathscr{L}$, we have that $\M \preceq \mathscr{L}$ and  $\mathscr{L}  \npreceq \M$.  
\end{Obs}

\begin{Lem} \label{Blade Runner 2049}
If  $\M \prec \mathscr{L} $, $\ds \lim_{t \to +\infty} \ds \frac{\left[ \ds \sup_{p \in \N_{0}} \left( \frac{t^p}{\ell_{p} \cdot p!} \right) \right]}{\left[\ds \sup_{p \in \N_{0}} \left( \frac{\delta^p \cdot  t^p}{m_{p} \cdot p!} \right) \right]} = 0$, for any $\delta > 0$. 
\end{Lem}
\begin{proof}
Given $\delta > 0$, take $H$ as in \eqref{Prisoners} and $q \in \N$ satisfying $ \ds \frac{1}{H^{q}} \leq \delta$.  Note that by hypothesis, $\ds\lim_{k \to +\infty} \left(\ds \frac{\ell_{k}}{m_{k}} \right)^{1/k} = + \infty$, and thus there exists $k_{0} \in \N$ such that 
\begin{equation} \label{Harry Potter and the Goblet of Fire}
\left(\ds \frac{\ell_{k}}{m_{k}} \right)^{1/k} \geq H^{q + 1}, \ \ \forall k \geq k_{0}. 
\end{equation}
Consider $t \geq \ell_{k_{0}} \cdot (k_{0}!)$; when $s < k_0$, 
\begin{equation} \label{Goodfellas} 
  \ds \frac{t^{k_{0}}}{\ell_{k_{0}} \cdot k_{0}!} \div \ds \frac{t^{s}}{\ell_{s} \cdot s!} = t^{k_{0} -s} \cdot \ds \frac{\ell_{s} \cdot s!}{\ell_{k_{0}} \cdot k_{0}!} \geq [\ell_{k_{0}} \cdot k_{0}!]^{k_{0} - s - 1} \cdot \ell_{s} \cdot s! \geq  1.
\end{equation}
 
Hence, if $ t \geq \ell_{k_{0}} \cdot (k_{0}!)$, it follows from \eqref{Harry Potter and the Goblet of Fire} and \eqref{Goodfellas} that  
\begin{equation} \label{Harry Potter and the Order of the Phoenix}
\ds \frac{\left[\ds \sup_{p \in \N_0} \left(\ds \frac{t^{p}}{\ell_{p} \cdot p!} \right) \right]}{\left[\ds \sup_{p \in \N_{0}} \left(\ds \frac{(\delta t)^{p}}{m_{p} \cdot p!} \right) \right]} = \ds \frac{\left[\ds \sup_{p \geq k_{0}} \left(\ds \frac{t^{p}}{\ell_{p} \cdot p!} \right) \right]}{\left[\ds \sup_{p \in \N_{0}} \left(\ds \frac{(\delta t)^{p}}{m_{p} \cdot p!} \right) \right]}  \leq \ds \frac{\left[\ds \sup_{p \in \N_{0}} \left(\ds \frac{t^{p}}{(H^{q + 1})^{p} \cdot m_{p} \cdot p!} \right) \right]}{\left[\ds \sup_{p \in \N_{0}} \left(\ds \frac{(\delta t)^{p}}{m_{p} \cdot p!} \right) \right]} .    
\end{equation}
By Lemma \ref{Jaws}, 
\begin{equation} \label{Harry Potter and the Half-Blood Prince}
\left[\ds \sup_{p \in \N_{0}} \left(\ds \frac{t^{p}}{(H^{q + 1})^{p} \cdot m_{p} \cdot p!} \right) \right] \leq \left[\ds \sup_{p \in \N_{0}} \left(\ds \frac{t^{p}}{(H^{q})^{p} \cdot m_{p} \cdot p!} \right) \right]^{1/2} \leq \left[\ds \sup_{p \in \N_{0}} \left(\ds \frac{\delta^{p} \cdot t^{p}}{m_{p} \cdot p!} \right) \right]^{1/2}.
\end{equation}
From \eqref{Harry Potter and the Order of the Phoenix} and \eqref{Harry Potter and the Half-Blood Prince}, we infer that 
$$\ds \frac{\left[ \ds \sup_{p \in \N_{0}} \left( \frac{t^p}{\ell_{p} \cdot p!} \right) \right]}{\left[\ds \sup_{p \in \N_{0}} \left( \frac{\delta^p \cdot  t^p}{m_{p} \cdot p!} \right) \right]} \leq \left[\ds \sup_{p \in \N_{0}} \left(\ds \frac{\delta^{p} \cdot t^{p}}{m_{p} \cdot p!} \right) \right]^{-1/2} \leq \ds \frac{\sqrt{2 \cdot m_2}}{\delta t}, \ \ \forall t \geq \ell_{k_0} \cdot (k_{0})!, $$
which proves the assertion. 
\end{proof}

\begin{Teo} \label{Road to Perdition} 
Let $\mathscr{L}$, $\M$ be weight sequences such that $\M \prec \mathscr{L}$ and  $P_{\alpha}$ as vector field as in  $\eqref{The Shawshank Redemption}$. 
\begin{enumerate} [leftmargin=*]
 \item \label{Harry Potter and the Deathly Hallows} When $P_{\alpha}$ is globally $\mathscr{L}$-hypoelliptic, it is also globally $\M$-hypoelliptic. 
 \item \label{Mission Impossible 2} There exists  $\beta \in \R \setminus \Q$ such that $P_{\beta}$ is globally $\mathscr{\M}$-hypoelliptic, but \textbf{not} globally $\mathscr{L}$-hypoelliptic. 
\end{enumerate}

\end{Teo}

\begin{proof}
The proof of \ref{Harry Potter and the Deathly Hallows} follows immediately from Corollary  \ref{Harry Potter and the Chamber of Secrets} and Remark \ref{Mission Impossible: Ghost Protocol}. In order to prove \ref{Mission Impossible 2}, we will use \cite{hardy1979introduction} as reference for theory of continued fractions and apply an argument based on \cite{yoshino} to exhibit $\beta = [a_{0}, a_{1},\ldots, a_{n}, \ldots ]$ in the interval $(0,1)$ satisfying the conditions required, defining the sequence $\left\{a_{n} \right\}_{n \in \N_{0}}$ recursively. 

Following  \cite{hardy1979introduction} (Teo. 149), we introduce  $\left\{p_{n} \right\}_{n \in \N_{0}}, \left\{q_{n} \right\}_{n \in \N_{0}}$ as
\begin{equation} \label{Remember the Titans}
p_{0} = 0, \ \ \ \ \ p_{1} = 1, \ \ \ \ \ p_{n} = a_{n} \cdot p_{n-1} + p_{n-2} \ \ \ (2 \leq n).   
\end{equation}
\begin{equation} \label{Man on Fire}
q_{0} = 1, \ \ \ \ \ q_{1} = 0, \ \ \ \ \ q_{n} = a_{n} \cdot q_{n-1} + q_{n-2} \ \ \ (2 \leq n).   
\end{equation}

\noindent We also denote (Sec. 10.9)  $\left\{a_{n}'\right\}_{n \in \N_{0}}$ as 

$$a'_{n} = [a_{n}, a_{n+1}, \ldots], \ \ \forall n \in \N_{0}.$$  

\noindent The following results hold: 

\begin{enumerate} [leftmargin=*]
 \item \label{Gladiator} (Thm 155, 156).   For $n > 3$, $q_{n+1} \geq q_{n} \geq n$. Thus $\ds \lim_{n \to \infty} q_{n} = + \infty $. 
 \item \label{American Gangster}(Thm 168). For every $n \in \N_{0}$, $\left\lfloor a_{n}'\right\rfloor = a_{n}$, where $\left\lfloor {.} \right\rfloor$ is the floor function.   
 \item \label{Black Hawk Down} (Thm 171). For every $n \geq 1$, $|p_{n} - \beta \cdot q_{n}| = \ds \frac{1}{a_{n+1}' \cdot q_{n} + q_{n-1}}$. Hence $|p_{n} - \beta \cdot q_{n}|$ is strictly decreasing and tends to $0$. 
 \item \label{The Martian}(Thm 182). Let  $p \in \Z$, $q \in \N$ such that  $gcd \ (p,q) = 1$ and $q_{k} \leq q < q_{k+1}$. Then   
$$|p - q \cdot \beta | \geq |p_{k} - q_{k} \cdot \beta| > |p_{k+1} - q_{k+1} \cdot \beta|.$$

\end{enumerate}

Consider $a_{0} = 0$  and suppose $a_{j}$ set for  $0 \leq j \leq n-1$. It follows from \eqref{Remember the Titans} and \eqref{Man on Fire} that $p_{j}, q_{j}$ are well defined for $j \leq n-1$.  We take then 

$$
a_{n} = 
\begin{cases}
\left\lfloor \ds \sup_{r \in \N_0} \left(\ds \frac{(q_{n-1} + 1)^{r}}{\ell_{r} \cdot r!} \right) \Big{/} q_{n-1}  \right\rfloor, \  &n \neq 2. \\
1, \ &n = 2.
\end{cases}
$$

Let us show that for $\beta = \left[a_0, a_1, \ldots, a_n, \ldots \right]$ $P_{\beta}$ is globally  $\mathscr{L}$-hypoelliptic. Put $p \in \Z$, $q \in \N$ such that $gcd \ (p,q) = 1$. From \ref{Gladiator}., we conclude the existence of $k_{0} \in \N$ such that $q_{k_{0}} \leq q < q_{k_{0} + 1}$. From \ref{Black Hawk Down}. and \ref{The Martian}., we obtain: 
\begin{equation} \label{Philadelphia}
|p - q \cdot \beta | \geq |p_{k_{0}} - q_{k_{0}} \cdot \beta| = \ds \frac{1}{a_{k_{0}+1}' \cdot q_{k_{0}} + q_{k_{0} - 1}}.  
\end{equation}

On the other hand, if $q \geq q_{5}$, we infer from \ref{American Gangster}. and definition of $\left\{a_{n} \right\}_{n \in \N_{0}}$ that 
\begin{align*}
a_{k_{0}+1}' \cdot q_{k_{0}} + q_{k_{0} - 1} &= \left\lfloor \ds \sup_{r \in \N_0} \left(\ds \frac{(q_{k_0} + 1)^{r}}{\ell_{r} \cdot r!} \right) \Big{/} q_{k_0}  \right\rfloor \cdot q_{k_{0}} + q_{k_{0} - 1} \\
& \leq 2 \ds \sup_{r \in \N_0} \left(\ds \frac{(q + 1)^{r}}{\ell_{r} \cdot r!} \right).
\end{align*}
Hence, by \eqref{Philadelphia},
\begin{equation} \label{Fences}
|p - q \cdot \beta| \geq \ds \frac{1}{2} \cdot \ds \inf_{r \in \N_0} \left(\ds \frac{\ell_{r} \cdot r!} {(q + 1)^{r}} \right). 
\end{equation}

We now fix $\delta > 0$; it follows from  Lemma \ref{Blade Runner 2049} that  $\ds \lim_{t \to +\infty} \ds \frac{\ds \inf_{r \in \N_0} \left(\ds \frac{m_{r} \cdot r!} {\delta^{r} \cdot t^{r}} \right)}{\ds \inf_{r \in \N_0} \left(\ds \frac{\ell_{r} \cdot r!} {t^{r}} \right)} = 0$. Thus there exists $s \in \N$, which we may consider $s \geq q_{5}$, such that  

$$\ds \inf_{r \in \N_0} \left(\ds \frac{ m_{r} \cdot r!} {\delta^{r} \cdot t^{r}} \right) \leq \ds \frac{1}{2} \cdot \ds \inf_{r \in \N_0} \left(\ds \frac{\ell_{r} \cdot r!} {t^{r}} \right), \ \ \ \ \ t \geq s.$$

\noindent So it follows from \eqref{Fences} and the estimate above that
\begin{equation} \label{Hurricane}
|p - q \cdot \beta| \geq \ds \inf_{r \in \N_0} \left(\ds \frac{ m_{r} \cdot r!} {\delta^{r} \cdot (1+q)^{r}} \right), \ \ \forall q \geq s. 
\end{equation}

Consider now that $|p| + |q| \geq 2s$. If $q \geq s$, we apply \eqref{Hurricane}. Otherwise,  
$$|p - q \cdot \beta| \geq |p| - |q| \geq (s+1) - (s-1) = 2 >  \ds \inf_{r \in \N_0} \left(\ds \frac{ m_{r} \cdot r!} {\delta^{r} \cdot (1+q)^{r}} \right). $$
Therefore, by Proposition \ref{The Silence of the Lambs} we deduce that $P_{\beta}$ is globally $\M$-hypoelliptic.  

To prove the second part, let us estimate $|p_{k_{0}} - q_{k_{0}} \cdot \beta|$ from below:
\begin{align*}
a_{k_{0}+1}' \cdot q_{k_{0}} + q_{k_{0} - 1} &= \left\lfloor \ds \sup_{r \in \N_0} \left(\ds \frac{(q_{k_0} + 1)^{r}}{\ell_{r} \cdot r!} \right) \Big{/} q_{k_0}  \right\rfloor \cdot q_{k_{0}} + q_{k_{0} - 1} \\
&> \ds \sup_{r \in \N_0} \left(\ds \frac{(q_{k_0} + 1)^{r}}{\ell_{r} \cdot r!} \right) - q_{k_{0}}.
\end{align*}
With a very similar argument to the one applied in Lemma \ref{Blade Runner 2049}, one can prove that 
$$\ds \lim_{t \to + \infty} \ds \frac{\ds \sup_{r \in \N_0} \left(\ds \frac{t^{r}}{\ell_{r} \cdot r!} \right)}{t} = + \infty.$$ 

Thence, it follows from \ref{Gladiator} that we can find $d \in \N$ such that if $k_{0} \geq d$,  

$$a_{k_{0}+1}' \cdot q_{k_{0}} + q_{k_{0} - 1} > \ds \frac{1}{2} \cdot \ds \sup_{r \in \N_0} \left(\ds \frac{(q_{k_{0}} + 1)^{r}}{\ell_{r} \cdot r!} \right). $$

\noindent That is, 
\begin{equation} \label{Wild Strawberries}
|p_{n} - q_{n} \cdot \beta| < 2 \cdot \ds \inf_{r \in \N_0} \left(\ds \frac{\ell_{r} \cdot r!} {(1 + q_n)^{r}} \right), \ \ \ n \geq d. 
\end{equation}

\noindent It is not difficult to see that \eqref{Wild Strawberries} shows us that $\beta$ is Liouville  $\mathscr{L}$-exponential, as we intended to prove.  
\end{proof}

\subsection[Normal Form]{Global $\M$-hypoellipticity for a Class of Systems of Real Vector Fields}

In this subsection, our environment will be the $(N+1)$ dimensional torus, denoted as $\T^{N+1}$, for some natural number $N$. We will denote its elements as $(t,x)$, with $t \in \T^{N}$, $x \in \T$ and consider the following system of vector fields:

\begin{equation} \label{It's a Wonderful Life}
L_{j} = \ds \frac{\partial}{\partial t_{j}} + a_{j} (t_{j}) \cdot \ds \frac{\partial}{\partial x}, \ \ \ \ \ \ \ (j = 1, 2, \ldots, N), 
\end{equation}
where each $a_j$ is a real-valued element of $\E_{\M}(\T)$. 

In a similar way to what was  done in \cite{hounie}(smooth case) and \cite{arias} (Gevrey case),  we intend to show that \eqref{It's a Wonderful Life} and
\begin{equation} \label{2001} 
\tilde{L}_{j} = \ds \frac{\partial}{\partial t_{j}} + a_{j_{0}} \cdot \ds \frac{\partial}{\partial x}, \ \ \ \ \ \ \ (j = 1, 2, \ldots, N),
\end{equation}
where $a_{j_{0}} = \left[\ds \frac{1}{2 \pi} \cdot \ds \int_{0}^{2 \pi} a_{j}(s) \ ds \right]$ is $a_{j}$'s average, are equivalent in terms of global $\M$-hypoellipticity. We start stating some technical results.

\begin{Lem} \label{Casablanca} (Proposition 4.4 of \cite{bierstone2004resolution})
Let $n$ be a natural number and $k_{1}, \ldots, k_{n}$ non-negative integers such that $k_{1} + 2k_{2} + \ldots + nk_{n} = n$. Then, for  $k:= k_{1} + k_{2} + \ldots + k_{n}$, we have that
$$m_{k} \cdot \ds \prod_{\ell=1}^{n} (m_{\ell})^{k_{\ell}} \leq m_{n}.$$
\end{Lem}

\begin{Lem} \label{Lawrence of Arabia} (Lemma 1.4.1 of \cite{krantz2002primer})
For each positive integer $n$ and positive real number $R$, 
$$\ds \sum \ds \frac{k!}{k_{1}! \cdot  k_{2}! \cdot  \ldots \cdot k_{n}!} \cdot R^{k} = R \cdot (1+R)^{n-1} $$
\noindent  where $k = k_1 + k_2 + \ldots + k_n$ and the sum is taken over all $k_1, k_2, \ldots, k_n$ for which $k_1 + 2k_2 + \ldots + n k_n = n$. 
\end{Lem}

Since a periodic function has periodic primitive if and only if it has null average, the key step will be to take the primitive of each $a_{j}$ minus its average and to sum all of them. 

\begin{Def}
Let $a_{j}(t_{j}) \in \E_{\M}(\T^N)$ for $j = 1, 2, \ldots, N$ and $a_{j_0}$ its respective average.  We define  
\begin{equation} \label{Bonnie and Clyde}
A: \T^{N} \to \R; \ \ \  A(t) = \ds \sum_{j=1}^{N} \left[\ds \int_{0}^{t_{j}} a_{j}(s)  ds - a_{j_{0}} \cdot t_{j}  \right].
\end{equation}
\end{Def}

\begin{Pro} \label{Modern Times}
Let $A$ be as in  \eqref{Bonnie and Clyde};  for every $\varepsilon > 0$, there exist $C_{\varepsilon}, h_{\varepsilon} > 0$ such that
$$\ds \inf_{p \in \N_{0}} \left( \ds \frac{m_{p} \cdot p!}{\varepsilon^{p} \cdot (1+|\eta|)^{p} } \right) \cdot \left|\partial_{t}^{\alpha} e^{i \eta A(t)} \right| \leq C_{\varepsilon} \cdot h_{\varepsilon}^{|\alpha|} \cdot m_{|\alpha|} \cdot |\alpha|!, \ \ \ \ \ \forall t \in \T^N, \ \forall \eta \in \Z, \ \forall \alpha \in \N_{0}^{N}.$$ 
\end{Pro}
\begin{proof}
By denoting, for each $j$, $A_{j}(t_{j}) = \ds \int_{0}^{t_{j}} a_{j}(s)  ds - a_{j_{0}} \cdot t_{j},$ we have $A(t) = \ds \sum_{j=1}^{N} A_{j}(t_{j})$. In this situation,
\begin{equation} \label{City Lights}
\left|\partial_{t}^{\alpha} e^{i \eta A(t)} \right| =  \ds \prod_{j=1}^{N} \left| \partial_{t_{j}}^{\alpha_{j}} \left( e^{i \eta \cdot  A_{j}(t_{j})} \right) \right|
\end{equation}
Because every $A_{j}(t_{j})$ belongs to $\E_{\M}(\T)$, one can find $C, h > 1$ satisfying
\begin{equation} \label{Raging Bull}
\left|\partial_{t_j}^{\ell} A_{j}(t_{j}) \right| \leq C \cdot h^{\ell} \cdot m_{\ell} \cdot \ell!, \ \ \ \  \forall t \in \T, \ \ \forall \ell \in \N_{0}, \ \ \forall j \in \left\{1, 2, \ldots, n \right\}. 
\end{equation}

For any $\alpha \in \N$, we define the following set:
$$\Delta(\alpha) = \left\{(k_{1}, k_{2}, \ldots, k_{\alpha}) \in \N_{0}^{\alpha}; \ \ k_{1} + 2 \cdot k_{2} + \ldots + \alpha \cdot k_{\alpha}  = \alpha \right\}.$$
Then, we apply  Fa\`a di Bruno's formula (see \cite{bierstone2004resolution} for a proof) associated to  \eqref{Raging Bull}:
\begin{align*}
\left|\partial_{t_{j}}^{\alpha_{j}} \left(e^{i \cdot \eta \cdot A_{j}(t_{j})} \right) \right| &= \left| \ds \sum_{\Delta(\alpha_{j})} \ds \frac{\alpha_{j} !}{k_{1}! \cdot k_{2} \cdot \ldots k_{\alpha_{j}}!} \cdot e^{i \cdot \eta \cdot A_{j}(t_{j})} \cdot \ds \prod_{\ell = 1}^{\alpha_{j}} \left[ \ds \frac{\partial_{t_{j}}^{\ell} (i  \cdot \eta \cdot A_{j}(t_{j}))}{\ell!} \right]^{k_{\ell}} \right|  \\
&\leq  \ds \sum_{\Delta(\alpha_{j})} \ds \frac{\alpha_{j} !}{k_{1}! \cdot k_{2} \cdot \ldots k_{\alpha_{j}}!} \cdot |\eta|^{k} \cdot \ds \prod_{\ell = 1}^{\alpha_{j}} \left[ \ds \frac{C \cdot h^{\ell} \cdot m_{\ell} \cdot \ell!}{\ell!} \right]^{k_{\ell}}, 
\end{align*}
for $k := \ds \sum_{\ell = 1}^{\alpha_{j}} k_{\ell}$. Given $\sigma > 0$, we now consider $(\star) := \ds \inf_{p \in \N_{0}} \left(\ds \frac{m_{p} \cdot p!}{\sigma^{p} \cdot (1+|\eta|)^{p}} \right) \cdot \left|\partial_{t}^{\alpha} e^{i \eta A(t)} \right|. $
So 
\begin{align*}
(\star) &\leq  \ds \sum_{\Delta(\alpha_{j})} \ds \frac{\alpha_{j} !}{k_{1}! \cdot k_{2} \cdot \ldots k_{\alpha_{j}}!} \cdot \left(\ds \frac{m_{k} \cdot k!}{\sigma^{k} \cdot (1+|\eta|)^{k}} \right) \cdot (1+|\eta|)^{k} \cdot  \ds \prod_{\ell = 1}^{\alpha_{j}} \left[C \cdot h^{\ell} \cdot m_{\ell} \right]^{k_{\ell}} \\
&\leq  \ds \sum_{\Delta(\alpha_{j})} \ds \frac{\alpha_{j} !}{k_{1}! \cdot k_{2} \cdot \ldots k_{\alpha_{j}}!} \cdot \left(\ds \frac{C^{k} \cdot h^{\alpha_{j}} \cdot  k!}{\sigma^{k} } \right) \cdot  \ds \prod_{\ell = 1}^{\alpha_{j}} m_{k} \cdot m_{\ell} ^{k_{\ell}}
\end{align*}

Applying Lemmas \ref{Casablanca} and \ref{Lawrence of Arabia}, we infer that
\begin{align*}
(\star) &\leq  \ds \sum_{\Delta(\alpha_{j})} \ds \frac{\alpha_{j} !}{k_{1}! \cdot k_{2} \cdot \ldots k_{\alpha_{j}}!} \cdot \left(\ds \frac{C^{k} \cdot h^{\alpha_{j}} \cdot k!}{\sigma^{k} } \right) \cdot m_{\alpha_{j}} \\
&\leq (h^{\alpha_{j}} \cdot m_{\alpha_{j}} \cdot \alpha_{j}!) \cdot \left[\ds \sum_{\Delta(\alpha_{j})} \ds \frac{k!}{k_{1}! \cdot k_{2} \cdot \ldots k_{\alpha_{j}}!} \cdot \left(\ds \frac{C}{\sigma} \right)^{k} \right] \\
&\leq \ds \frac{C}{\sigma} \cdot \left[\left(1 + \ds \frac{C}{\sigma} \right) \cdot h \right]^{\alpha_{j}} \cdot m_{\alpha_{j}} \cdot \alpha_{j}!.
\end{align*}

\noindent By taking $C_{\sigma} = \ds \frac{C}{\sigma}$ and $h_{\sigma} = \left[\left(1 + \ds \frac{C}{\sigma} \right) \cdot h \right]$, we conclude that

\begin{equation} \label{The Age of Innocence}
\ds \inf_{p \in \N_{0}} \left(\ds \frac{m_{p} \cdot p!}{\sigma^{p} \cdot (1+|\eta|)^{p}} \right) \cdot \left|\partial_{t}^{\alpha_{j}} e^{i \eta A_{j}(t_{j})} \right| \leq  C_{\sigma} \cdot h_{\sigma}^{\alpha_{j}} \cdot m_{\alpha_{j}} \cdot \alpha_{j}!, \ \ \ j \in \left\{1, 2, \ldots , N \right\}.
\end{equation}

Using Lemma \ref{Jaws}, one can prove by induction  that 

\begin{equation} \label{Atonement} 
\left[\ds \sup_{n \in \N_{0}} \left( \ds \frac{\rho^{n}}{m_{n} \cdot n!} \right) \right]^{2^{k}} \leq \ds \sup_{n \in \N_{0}} \left( \ds \frac{\rho^{n} \cdot (H^{k})^n}{m_{n} \cdot n!} \right), \ \ \ \forall \rho > 0, \ \ \forall k \in \N_{0}.  
\end{equation}
So we choose $k \in \N$ such that  $2^{k} \geq N$ and $\sigma = \ds \frac{\varepsilon}{H^{k}}$. Using \eqref{City Lights} and \eqref{The Age of Innocence},  we deduce that

\begin{equation} \label{One Million Dollar Baby}
\ds \inf_{p \in \N_{0}} \left(\ds \frac{m_{p} \cdot p!}{\sigma^{p} \cdot (1+|\eta|)^{p}} \right)^{N} \cdot \left|\partial_{t}^{\alpha} e^{i \eta A(t)} \right| \leq  C_{\sigma}^{N} \cdot h_{\sigma}^{|\alpha|} \cdot m_{|\alpha|} \cdot |\alpha|!. 
\end{equation}

On the other hand, since $2^{k} \geq N$, we use \eqref{Atonement} in order to obtain 

$$\ds \inf_{p \in \N_{0}} \left(\ds \frac{m_{p} \cdot p!}{\sigma^{p} \cdot (1+|\eta|)^{p}} \right)^{N} \geq \ds \inf_{p \in \N_{0}} \left(\ds \frac{m_{p} \cdot p!}{\varepsilon^{p} \cdot (1+|\eta|)^{p}} \right).$$

\noindent Hence, by associating \eqref{One Million Dollar Baby} to the inequality above and taking $C_{\varepsilon} = C_{\sigma}^N, \ h_{\varepsilon} = h_{\sigma}$, we have

$$\ds \inf_{p \in \N_{0}} \left(\ds \frac{m_{p} \cdot p!}{\varepsilon^{p} \cdot (1+|\eta|)^{p}} \right) \cdot \left|\partial_{t}^{\alpha} e^{i \eta A(t)} \right| \leq  C_{\varepsilon}^{N} \cdot h_{\varepsilon}^{|\alpha|} \cdot m_{|\alpha|} \cdot |\alpha|!,$$

\noindent as we intended to prove. 
\end{proof}

\begin{Teo} \label{Butch Cassidy and the Sundance Kid}
Let $T: \D_{\M}'(\T^{N+1}) \to \D_{\M}'(\T^{N+1})$ be the operator given by 
$$T \left(\ds \sum_{\eta \in \Z} \widehat{u}(t, \eta) e^{i x \eta} \right) = \ds \sum_{\eta \in \Z} \widehat{u}(t, \eta) e^{i (A(t) + x) \eta}. $$
Then $T$ is an automorphism. Moreover, the same holds for $T \big{|}_{\E_{\M}(\T^{N+1})}$.
\end{Teo}

\begin{proof}
We verify first that $T$ is well defined. By Theorem \ref{A Clockwork Orange}, given $\varepsilon, h > 0$, one has to find $C_{\varepsilon, h} > 0$ such that 
$$ \left|\<\widehat{u}(t, \eta) e^{i \eta A(t)}, \psi \> \right| \leq C_{\varepsilon, h} \cdot \left\|\psi \right\|_{\M, h} \cdot \ds \sup_{n \in \N_{0}} \left( \ds \frac{\varepsilon^{n} \cdot (1+|\eta|)^{n}}{m_{n} \cdot n!} \right), \ \ \forall \eta \in \Z, \ \ \forall \psi \in \E_{\M, h}(\T^{N}). $$

Let us fix $\varphi \in \E_{\M, h}(\T^{N})$; for $\delta > 0$ that will be chosen later, we denote

$$(\triangle) = \ds \inf_{p \in \N_{0}} \left(\ds \frac{m_{p} \cdot p!}{\delta^{p} \cdot (1+|\eta|)^{p}} \right) \cdot \left |\partial_{t}^{\alpha} \left(e^{i \eta A(t)} \cdot \varphi(t) \right) \right|.$$

\noindent By Proposition \ref{Modern Times}, 
\begin{align*}
(\triangle)  
&\leq \ds \sum_{\beta \leq \alpha} \binom {\alpha} {\beta} \left( C_{\delta} \cdot h_{\delta}^{|\beta|} \cdot m_{|\beta|} \cdot |\beta|! \right) \cdot \left(\left\|\varphi \right\|_{\M, h} \cdot h^{|\alpha| - |\beta|} \cdot m_{|\alpha| - |\beta|} \cdot (|\alpha| - |\beta|)! \right) \\
&\leq C_{\delta} \cdot \left\|\varphi \right\|_{\M, h} \cdot h_{\delta}'^{|\alpha|} \cdot m_{|\alpha|} \cdot |\alpha|!,
\end{align*}
where $h'_{\delta} = 2 \cdot \max \left\{h_{\delta}, h \right\}$. Therefore
\begin{equation} \label{Spartacus}
\ds \inf_{p \in \N_{0}} \left(\ds \frac{m_{p} \cdot p!}{\delta^{p} \cdot (1+|\eta|)^{p}} \right) \cdot \left |  \partial_{t}^{\alpha} \left(e^{i \eta A(t)} \cdot \varphi(t) \right) \right| \leq C_{\delta} \cdot \left\|\varphi \right\|_{\M, h} \cdot h_{\delta}'^{|\alpha|} \cdot m_{\alpha} \cdot \alpha!.
\end{equation}

On the other hand, for every $\rho > 0$, one can find $C_{\rho} > 0$ satisfying
$$\left|\<u, f \> \right| \leq C_\rho \cdot \underset{\gamma \in \N_{0}^{N+1}}{\ds \sup_{{x \in \T^{N+1}}}}  \left(\ds \frac{|\partial^\gamma f (x)| \cdot \rho^{|\gamma|}} {m_{|\gamma|} \cdot |\gamma|!} \right), \ \ \ \ \ \forall f \in \E_\M(\T^{N+1}), $$
by Theorem \ref{The Hurt Locker}. So
\begin{align} \label{Jerry Maguire}
\left|\<\widehat{u}(t, \eta) e^{i \eta A(t)}, \varphi(t) \> \right| &\leq  \ds \frac{C_\rho}{2 \pi} \cdot \underset{(\gamma_{1}, \gamma_{2}) \in \N_{0}^{N+1}}{\ds \sup_{{(t,x) \in \T^{N+1}}}}  \ds \frac{| \partial_{t}^{\gamma_{1}} \left(e^{i \eta A(t) } \varphi(t) \right)| \cdot |\partial_{x}^{\gamma_{2}} \left(e^{-i x \eta} \right)| \cdot \rho^{|\gamma_{1}| + |\gamma_{2}|}} {m_{|\gamma_{1}| + |\gamma_{2}|} \cdot (|\gamma_{1}| + |\gamma_{2}|)!} \nonumber\\ 
&\leq  \ds \frac{C_\rho}{2 \pi} \cdot \underset{(\gamma_{1}, \gamma_{2}) \in \N_{0}^{N+1}}{\ds \sup_{{(t,x) \in \T^{N+1}}}}  \ds \frac{| \partial_{t}^{\gamma_{1}} \left(e^{i \eta A(t) } \varphi(t) \right)| \cdot (1+ |\eta|)^{|\gamma_{2}|} \cdot \rho^{|\gamma_{1}| + |\gamma_{2}|}} {m_{|\gamma_{1}| + |\gamma_{2}|} \cdot (|\gamma_{1}| + |\gamma_{2}|)!}.
\end{align}

We now denote $(\star) = | \partial_{t}^{\gamma_{1}} \left(e^{i \eta A(t) } \varphi(t) \right)| \cdot (1+ |\eta|)^{|\gamma_{2}|} \cdot \rho^{|\gamma_{1}| + |\gamma_{2}|}$.  By Lemma \ref{Jaws}, 
 
$$\left(\ds \sup_{n \in \N_{0}} \ds \frac{\varepsilon^{n} \cdot (1+|\eta|)^{n}}{H^{n} \cdot m_{n} \cdot n!} \right)^{2} \leq \left(\ds \sup_{n \in \N_{0}} \ds \frac{\varepsilon^{n} \cdot (1+|\eta|)^{n}}{m_{n} \cdot n!} \right).$$
 
\noindent Then, if $\delta = \ds \frac{\varepsilon}{H}$, 
\begin{align*}
(\star) &\leq \left|\partial_{t}^{\gamma_{1}} \left(e^{i A(t) \eta} \varphi(t) \right)\right| \cdot  \ds \inf_{p \in \N_{0}} \left(\ds \frac{m_{p} \cdot p!}{\delta^{p} \cdot (1+|\eta|)^{p}} \right)^{2} \cdot  (1+ |\eta|)^{|\gamma_{2}|} \cdot  \left(\ds \sup_{n \in \N_{0}} \ds \frac{\varepsilon^{n} \cdot (1+|\eta|)^{n}}{m_{n} \cdot n!} \right) \cdot \rho^{|\gamma_{1}| + |\gamma_{2}|}. 
\end{align*}

\noindent Applying \eqref{Spartacus}, we infer that
\begin{align*}
(\star) &\leq  C_{\delta} \cdot \left\|\varphi \right\|_{\M, h} \cdot h_{\delta}'^{|\gamma_{1}|} \cdot m_{|\gamma_{1}|} \cdot |\gamma_{1}|! \cdot  \left(\ds \frac{m_{|\gamma_{2}|} \cdot |\gamma_{2}|! \cdot H^{|\gamma_{2}|}}{\varepsilon^{|\gamma_{2}|}}  \right) \cdot \left(\ds \sup_{n \in \N_{0}} \ds \frac{\varepsilon^{n} \cdot (1+|\eta|)^{n}}{m_{n} \cdot n!} \right) \cdot  \rho^{|\gamma_{1}| + |\gamma_{2}|}. 
\end{align*}

Because we may suppose $\varepsilon < H$, 
$$(\star) \leq C_{\delta} \cdot  \left\|\varphi \right\|_{\M, h} \cdot \left(\ds \frac{h_{\delta}' \cdot H \cdot \rho}{\varepsilon} \right)^{|\gamma_{1}| + |\gamma_{2}|} \cdot m_{|\gamma_{1}| + |\gamma_{2}|} \cdot (|\gamma_{1}| + |\gamma_{2}|)! \cdot  \left(\ds \sup_{n \in \N_{0}} \ds \frac{\varepsilon^{n} \cdot (1+|\eta|)^{n}}{m_{n} \cdot n!} \right).$$

\noindent Since the right-hand side of the inequality above does not depend on $(t,x)$, by taking $\rho = \ds \frac {\varepsilon}{h_{\delta}' \cdot H}$ and applying \eqref{Jerry Maguire}, we obtain that 

\begin{equation} \label{The Sixth Sense}
\left|\<\widehat{u}(t, \eta) e^{i \eta A(t)}, \varphi(t) \> \right| \leq  \ds \frac{C_\rho}{2 \pi} \cdot C_{\delta} \cdot  \left\|\varphi \right\|_{\M, h} \cdot \left(\ds \sup_{n \in \N_{0}} \ds \frac{\varepsilon^{n} \cdot (1+|\eta|)^{n}}{m_{n} \cdot n!} \right).
\end{equation}

\noindent Observe  that both $\rho$ and $\delta$ depend only on $\varepsilon$ and $h$, and thus the first part of the proof is completed.

Now let $\psi \in \E_{\M}(\T^{N+1})$; by Theorem \ref{Children of Men}, there exist $C, h, \delta > 0$ such that
$$|\partial_{t}^{\beta}\widehat{\psi}(t, \eta) | \leq C \cdot h^{|\beta|} \cdot m_{|\beta|} \cdot |\beta|! \cdot \left(\ds \inf_{p \in \N_{0}} \ds \frac{m_{p} \cdot p!}{\delta^{p} \cdot (1+|\eta|)^{p}} \right), \ \ \ \forall t \in \T, \ \ \forall \beta \in \N_{0}^{N}.$$
Hence
\begin{align} \label{Back to the Future}
\left| \partial_{t}^{\alpha} \left(\widehat{\psi}(t, \eta) \cdot e^{i A(t) \eta} \right) \right| &\leq \ds \sum_{\beta \leq \alpha} \binom {\alpha} {\beta} \left|\partial_{t}^{\beta}\widehat{\psi}(t, \eta)  \right| \cdot \left|\partial_{t}^{\alpha - \beta} (e^{i A(t) \eta})  \right| \nonumber \\
& \leq \ds \sum_{\beta \leq \alpha} \binom {\alpha} {\beta} C \cdot h^{|\beta|} \cdot m_{|\beta|} \cdot |\beta|! \cdot \left(\ds \inf_{p \in \N_{0}} \ds \frac{m_{p} \cdot p!}{\delta^{p} \cdot (1+|\eta|)^{p}} \right) \cdot \left|\partial_{t}^{\alpha - \beta} (e^{i A(t) \eta})  \right|
\end{align}

Applying once again Lemma \ref{Jaws} and \eqref{Spartacus}, we have
\begin{align} \label{Back to the Future II}
 \left(\ds \inf_{p \in \N_{0}} \ds \frac{m_{p} \cdot p!}{\delta^{p} \cdot (1+|\eta|)^{p}} \right) \left|\partial_{t}^{\alpha - \beta} (e^{i A(t) \eta})  \right| &\leq \left(\ds \inf_{p \in \N_{0}} \ds \frac{H^{p} \cdot m_{p} \cdot p!}{\delta^{p} \cdot (1+|\eta|)^{p}}\right)^{2} \cdot \left|\partial_{t}^{\alpha - \beta} (e^{i A(t) \eta})  \right| \nonumber \\
&\leq \left(\ds \inf_{p \in \N_{0}} \ds \frac{H^{p} \cdot m_{p} \cdot p!}{\delta^{p} \cdot (1+|\eta|)^{p}}\right) \cdot C_{1} \cdot h_{1}^{|\alpha| - |\beta|} \cdot m_{|\alpha| - |\beta|} \cdot (|\alpha| - |\beta|)!.  
\end{align}

\noindent Associating \eqref{Back to the Future} to \eqref{Back to the Future II}, we deduce that 
\begin{align*}
\left| \partial_{t}^{\alpha} \left(\widehat{\psi}(t, \eta) \cdot e^{i A(t) \eta} \right) \right| &\leq \ds \sum_{\beta \leq \alpha} \binom {\alpha} {\beta} C h^{|\beta|}  m_{|\beta|}  |\beta|!  \left(\ds \inf_{p \in \N_{0}} \ds \frac{H^{p} \cdot m_{p} \cdot p!}{\delta^{p} \cdot (1+|\eta|)^{p}}\right)  C_{1}  h_{1}^{|\alpha| - |\beta|}  m_{|\alpha| - |\beta|}  (|\alpha| - |\beta|)!  \\
&\leq C_{2} \cdot h_{2}^{|\alpha|} \cdot m_{|\alpha|} \cdot |\alpha|! \cdot \left(\ds \inf_{p \in \N_{0}} \ds \frac{H^{p} \cdot m_{p} \cdot p!}{\delta^{p} \cdot (1+|\eta|)^{p}}\right),
\end{align*}

\noindent by putting $C_{2} = C \cdot C_{1}$ and $h_{2} = 2 \cdot \max \left\{h_{1}, h \right\}$. Therefore, it follows from Theorem \ref{Children of Men} that  $T(\psi) \in \E_{\M} (\T^N)$. 

To see that the operator is an automorphism,  is easy to check that the same properties proved for $T$ still hold for $T': \D_{\M}(\T^{N+1}) \to \D_{\M}(\T^{N+1})$ given by
$$T' \left(\ds \sum_{\eta \in \Z} \widehat{u}(t, \eta) e^{i x \eta} \right) = \ds \sum_{\eta \in \Z} \widehat{u}(t, \eta) e^{i (-A(t) + x) \eta}.$$

\noindent and that $T'$ is actually the inverse of $T$ in $\D_{\M}'(\T^{N+1})$ and $\E_{\M}(\T^{N+1})$. 
\end{proof}

\begin{Teo} \label{Patton}
The system of complex vector fields described in \eqref{It's a Wonderful Life} is globally $\M$-hypoelliptic if and only if the same is valid for the system defined in \eqref{2001} . 
\end{Teo}

\begin{proof}
It follows immediately from Theorem \ref{Butch Cassidy and the Sundance Kid} and the fact that $\tilde{L}_{j} =  T \circ L_{j} \circ T^{-1}.$ 
\end{proof}

\begin{minipage}[b]{7cm}
	{\bf Alexandre Kirilov}\\
	Departament of Mathematics\\ 
	Federal University of Parana\'a\\
	Curitiba, PR 82590-300, Brazil\\
	E-mail: {\it alexandrekirilov@gmail.com}
\end{minipage}
\hfill
\begin{minipage}[b]{7cm}
	{\bf Bruno de Lessa Victor}\\
	Departament of Mathematics\\ 
	Federal University of Parana\'a\\
	Curitiba, PR 82590-300, Brazil\\
	E-mail: {\it brunodelessa@gmail.com}
\end{minipage}

\end{document}